\documentclass[12pt]{amsart}

\usepackage[utf8]{inputenc}     
\usepackage[T1]{fontenc}
\usepackage{lmodern}
\usepackage{graphicx}
\usepackage[english]{babel}
\usepackage{subfigure, enumitem}
\usepackage{amsmath, amsfonts, amssymb, amsthm,dsfont} 
\usepackage{hyperref} 
\usepackage[shortalphabetic]{amsrefs} 
\usepackage{stmaryrd, fancyhdr}
\usepackage[foot]{amsaddr}
\hypersetup{colorlinks = true, urlcolor = blue, bookmarksopen = true}

\def \equi#1{\mathrel{\mathop{\kern 0pt\sim}\limits_{#1}}}

\newcommand*{\qfrac}[2]{\ensuremath{#1 / \! \raisebox{ - .65ex}{\ensuremath{#2}}}}

\newcommand*{\legendre}[1]{\left(\frac{#1}{p}\right)}
\newcommand{\Matrix}[1]{\begin{pmatrix} #1 \end{pmatrix}}

\newtheorem{prop}{Proposition}
\newtheorem{cor}{Corollary}
\newtheorem{theo}{Theorem}
\newtheorem{lemme}{Lemma}

\theoremstyle{definition}
\newtheorem{defi}{Definition}
\newtheorem{rmq}{Remark}

\newcommand{\setC}{\mathbb{C}}

\newcommand{\setE}{\mathbb{E}}
\newcommand{\setF}{\mathbb{F}}

\newcommand{\setN}{\mathbb{N}}

\newcommand{\setQ}{\mathbb{Q}}
\newcommand{\setR}{\mathbb{R}}

\newcommand{\setZ}{\mathbb{Z}}

\newcommand{\E}{\mathcal{E}}

\newcommand{\calH}{\mathcal{H}}

\newcommand{\M}{\mathcal{M}}
\newcommand{\N}{\mathcal{N}}

\newcommand{\calS}{\mathcal{S}}

\numberwithin{equation}{section}

\title{Fourier coefficients of modular forms of half-integral weight in arithmetic progressions}

\author{\textsc{Corentin DARREYE}}
\address{Univ. Bordeaux, CNRS, Bordeaux INP, IMB, UMR 5251,  F-33400, Talence, France}
\email{corentin.darreye@u-bordeaux.fr}

\begin{document}
\renewcommand{\proofname}{Proof}
\renewcommand{\epsilon}{\varepsilon}
\renewcommand{\le}{\leqslant}
\renewcommand{\ge}{\geqslant}

\maketitle

\begin{abstract}

We study the probabilistic behavior of sums of Fourier coefficients in arithmetic progressions. We prove a result analogous to previous work  of  Fouvry-Ganguly-Kowalski-Michel and Kowalski-Ricotta  in the context of half-integral weight holomorphic cusp forms and for prime power modulus. We actually show that these sums follow in a suitable range a mixed Gaussian distribution which comes from the asymptotic mixed distribution of Salié sums. 

\end{abstract}

\tableofcontents

 \section{Introduction and statement of the results}
 \label{section1}
\vskip 0.5 cm \subsection{The framework}

Let $ f$ be a fixed cusp form of weight $\kappa$. In this work we are interested in the sums over arithmetic progressions of its normalized Fourier coefficients  $\hat f_\infty(n)$ at the cusp at infinity. Let $q$ be an integer and $w$ be a smooth non-zero real-valued function with compact support in $]0, +\infty[$. Let $X$ be a positive real number and put
 \begin{equation}
 \label{defS}
 S(X,q,a)=\sum_{n=a\:[q]}\hat f_\infty(n)w\left(\frac{n}{X}\right)
 \end{equation}
 for any integer $a$ coprime with $q$.\\

By the \textit{square root cancellation philosophy}, the size of $S(X,q,a)$ is expected to be bounded by $\sqrt{X/q}$ since the length of summation is roughly $X/q$ and the terms are bounded on average. This philosophy leads us to consider
\begin{equation}
\label{defE}
E(X,q,a)=\frac{S(X,q,a)}{\sqrt{X/q}}.
\end{equation}

 In \cite{FGKM},  a probabilistic study of such a quantity is performed when $q$ is a prime number which goes to infinity and for a more precise range of $q$ and $X$.  Precisely, the authors compute the moments of $E(X,q,a)$ and the introduction of the smooth cut off $w$ in \eqref{defS} makes all the moments converge. In particular, the following is proved.
 
 \begin{theo}[Fouvry, Ganguly, Kowalski, Michel \cite{FGKM}]
 \label{theo-FGKM}
 Let $p$ be an odd prime and $f$ be a Hecke eigenform of level 1 and integral weight $\kappa$. If $X$ satisfies $p^{2-\epsilon}\ll_\epsilon X=o(p^2)$ for any $\epsilon >0$, then the sequence of random variables
 $$\begin{array}{ccc}
 \setF_p ^* &\to&\setR\\
 a&\mapsto& E(X,p,a)
 \end{array}$$ 
 defined in \eqref{defE}, converges in law when $p\to+\infty$ to a Gaussian random variable of mean 0 and explicitly computable variance depending on the Petersson norm of f and the $L^2$-norm of w. 
 \end{theo}
 
 \begin{rmq} \
 \begin{enumerate}
 \item[$\bullet$] Actually \cite{FGKM} establishes the convergence for the ``natural'' error term 
 $$\frac{S(X,p,\cdot)-M_p}{\sqrt{X/p}}$$
 with $M_p = \frac{1}{p}\sum\limits_{n\ge 1}\hat f_\infty(n)w\left(\frac{n}{X}\right)$. But as mentioned in \cite{FGKM}, the main term $M_p$ decays rapidly to 0. 
 
 \item[$\bullet$]  The Fourier coefficients of $f$ are real up to a multiplicative factor since, in this case, the Hecke eigenvalues are real.
 
  \end{enumerate}
 \end{rmq}
 
 Analytic properties of sums of type \eqref{defS} with no smooth cutoff $w$ have been studied in many works before \cite{FGKM}. The evaluation of its variance when the weight $\kappa$ is an integer has drawn particular interest (see for example \cite{Bloomeraverage}, \cite{Luaverage} or \cite{LauZhao}). 
 
Obviously, this sum is interesting only for $q<X$ and in \cite{LauZhao}, it appears that the variance becomes very explicit in the range $X^{1/2}<q<X$ (the range $q< X^{1/2}$ remains more mysterious). Precisely, the authors in \cite{LauZhao} show that for $X^{1/2+ \epsilon}\ll q \ll X^{1-\epsilon}$,

\begin{equation}
\label{lauzhao}
\sum_{a=1}^q\Big|\!\!\!\sum_{\tiny\begin{array}{c}n\le X\\n=a\:[q]\end{array}}\hat f_\infty(n)\:\Big|^2\sim c_f X \phantom{eeeee} \text{ as } X\to+\infty
\end{equation}
and where $c_f$ is the residue at $s=1$ of the Rankin-Selberg $L$-function of $f$.

However, if $X^{1/4+\epsilon}\ll q \ll X^{1/2-\epsilon} $ then Lau and Zhao show that the sum in \eqref{lauzhao} is bounded up to a multiplicative constant by $qX^{1/2}$ which is smaller than $X$ for $q\ll X^{1/2-\epsilon}$.\\

 Actually, if one removes the smooth cutoff $w$ in the definition of $E(X,p,a)$ then Theorem \ref{theo-FGKM} still holds. This non-trivial fact has been shown in \cite{LesYesh}.\\
 
 Theorem \ref{theo-FGKM} has been generalized in \cite{KR} to any Hecke-Maass cusp form for the group $\text{GL}_d$ but in this case the Fourier coefficients may not be real so they can satisfy an asymptotic Gaussian distribution with complex values (where $\setC$ is identified with $\setR^2$). 
 
 Let us recall that if $f$ is a Hecke-Maass cusp form for $\text{GL}_d$ then it has Fourier coefficients of the form
 $$a_f(m_1,\ldots,m_{d-1})$$ 
 for any positive integers $m_1,\ldots,m_{d-2}$ and non-zero integer $m_{d-1}$. If, in this case, we let 
 $$\hat f_\infty(n)=a_f(n)=a_f(n,1,\ldots,1)$$  for any $n\ge1$, then these coefficients satisfy 
 $$\overline{a_f(m)}=a_{f^*}(m)$$ where $f^*$ is the dual of $f$. Thus, these coefficients are real if $f$ is self-dual \emph{i.e.} $f^*=f$. 
 
 The following is proved in \cite{KR}. 
 
 \begin{theo}[Kowalski, Ricotta \cite{KR}]
 \label{theo-KR}
 Let $p$ be an odd prime, $f$ be an even or odd ${\rm GL}_d$  Hecke-Maass cusp form and $\hat f_\infty(n)=a_f(n,1,\ldots,1)$ be its Fourier coefficients with $\hat f_\infty(1)=1$. If $X$ satisfies $p^{d-\epsilon}\ll_\epsilon X=o(p^d)$ for any $\epsilon >0$, then the sequence of random variables
 $$\begin{array}{ccc}
 \setF_p ^* &\to&\setC\\
 a&\mapsto&E(X,p,a)
 \end{array}$$ 
 defined in \eqref{defE}, converges in law when $p\to+\infty$ to a real (respectively complex) Gaussian random variable of mean 0 and explicitly computable covariance matrix depending on the Petersson norm of f and the $L^2$-norm of w if $f$ is self-dual (respectively not self-dual).
 \end{theo}
 
In the case of Theorem \ref{theo-KR} for $d>2$, the method in \cite{LesYesh} no longer works but one conjectures that the smooth cut-off in the definition of $E(X,p,a)$ can still be removed.
 
  \vskip 0.5 cm \subsection{Statement of the main results}
In both \cite{FGKM} and \cite{KR} the authors compute an asymptotic expansion of the $\nu$-th moment of $E(X,q,\cdot)$ and show that the main term is equal to
$$ \delta_{2\mid\nu}\frac{\nu !}{2^{\nu/2}(\nu/2)!}V^{\nu/2}$$
for a certain $V>0$, which is the moment of a Gaussian distribution of variance $V$.

The purpose of this paper is to compute those moments in the case of half-integral weight modular forms with modulus $q=p^N$. As we will see, the behavior we uncover in this case differs significantly from what is shown in \cite{FGKM} and \cite{KR}.

Our first main contribution focuses on the moments of $E(X,q,\cdot)$. In the following statement, the exact dependency of the variance with respect to the initial parameters is omitted. This issue will be dealt with later. 

\begin{theo} 
\label{theo-moment}
Let $f$ be a Hecke eigenform of level 4 and of half-integral weight $\ell+1/2$. Define $E(X,q,a)$ as in \eqref{defE}.
Let $\nu$ be a positive integer and $e\in\{\pm1\}$. There exists an explicit constant $C_\nu$ depending only
on $\nu$ such that for any odd prime power $q=p^N$, any $X>0$ and any $\epsilon >0$ with 
$$1 \le Y^{(1+\epsilon)2^{\nu-2}}<C_\nu q$$
 where $Y=4q^2/X$, we have
 \begin{align}
 \frac{2}{\varphi(q)}\!\!\!\sum_{\tiny\begin{array}{c}a\:[q]\\ \legendre{a}=e\end{array}}\!\!\!E(X,q,a)^\nu &=
   \delta_{2\mid\nu}\frac{\nu !}{(\nu/2)!}\left(\frac{1}{Y}\sum_{\tiny{\begin{array}{c}1\le m <Y^{1+\epsilon}\\ \left(\frac{m}{p}\right)=e\end{array}}}
 \hat f_0(m)^2B^2\left(\frac{m}{Y}\right)\right)^{\nu/2}\\
 &\phantom{aaaaaaaaaaaaaaaaaaaaaaaaaaaa} + O\left(Y^{-1/3 + \epsilon}+\frac{Y^{\nu/2+\epsilon}}{p}\right)
 \nonumber
 \end{align}
where $\hat f_0(n)$ is the $n$-th normalized Fourier coefficient at the cusp $0$ and $B$ is a smooth rapidly decreasing function depending only on $w$ and $\ell$. 
\end{theo}
 
 Actually we will prove a refined version of this theorem with more flexibility on the parameters (see Theorems \ref{theo-moment-impair} and \ref{theo-momentpair}). The fact that
  the remainder term is unbounded when $p$ is fixed and $Y\to+\infty$ is a first issue and we  will see that in this range, one can define from $S(X,q,a)$ an analogue of $E(X,q,a)$ which converges.   
 
On the contrary, if we want to have this error term negligible for any $\nu$, then $Y$ must be smaller than any power of $p$ so the Legendre symbol twisting the sum in the main term is a source of additional difficulty since controlling a short sum of large conductor is a very challenging problem. In fact, different behaviors can occur and we will explicitly highlight one of them in Proposition \ref{prop-Nx}. This matter will be discussed in detail in section 6 and we will deduce an analogue of Theorems \ref{theo-FGKM} and \ref{theo-KR} for  subsequences of $E(X,q,a)$.
 
  \begin{cor}
\label{cor-horizontal}
 Let $f$ be a Hecke eigenform of level 4 and of half-integral weight $\ell+1/2$. Assume that its Fourier coefficients are real. Define $E(X,q,a)$ as in \eqref{defE}. Let $(q_k)_{k\ge1}$ be any sequence of odd prime powers, say $q_k=p_k^{N_k}$, with $(p_k)_k$ the sequence of odd prime numbers. If $X_k$ is a function of $k$ satisfying
 \begin{equation}
 \label{conditions1}
 \frac{q_k^{2}}{X_k}=o(\log\log p_k) \text{\; and \;} X_k=o(q_k^2) \text{\;\;\;as $k\to+\infty$}
 \end{equation} 
 then there exists a subsequence of the random variables
 $$\begin{array}{ccc}
 \qfrac{\setZ}{q_k\setZ}^\times &\to&\setR\\
 a&\mapsto& E(X_k,q_k,a)
 \end{array}$$ 
 which converges in law to the mixed distribution
 $$\frac{1}{2}\delta_0 + \frac{1}{2}\N(0,2V_{f,w})$$
 where $\delta_0$ is the Dirac measure at $0$ and $\N(0,2V_{f,w})$ is a Gaussian distribution of mean $0$ and variance $2V_{f,w}=2\frac{(4\pi)^{\ell+1/2}}{\Gamma(\ell+1/2)}\langle f_0, f_0\rangle ||w||_2^2$ with $\langle f_0, f_0\rangle $ the
 square of the normalized Petersson norm of $f_0$.
 \end{cor}

\begin{rmq}\
\begin{itemize}
\item[$\bullet$] The growth condition \eqref{conditions1} can be improved to a larger range if we assume some classical conjectures. For example, under GRH it is enough to assume that $q_k^2/X_k = o(\log p_k).$\\

\item[$\bullet$] The assumption that the Fourier coefficients of $f$ are real is crucial here. However, it is a classical assumption that is made in many recent papers on this topic, notably the study of sign changes in the sequence $(\hat f_\infty(n))_{n\ge1}$ (see \cite{BruiKoh}, \cite{Hulse} or \cite{sign}). Examples of such forms are given in \cite[page 7]{BruiKoh} and \cite[page 109]{Cipra}.\\
\end{itemize}
\end{rmq}
This result is quite surprising when comparing it to the integral weight case. One may wonder why the convergence only holds for a subsequence of prime numbers $p$. We will see how this is a natural restriction using the  method of moments.

The appearence  of two distinct distributions in Corollary \ref{cor-horizontal} comes from the asymptotic mixed distribution of Salié sums and the difference of behavior when $a$ is a square modulo $q$ or not. That is why we decompose the moments of the random variables as the average of a moment over squares and a moment over non-squares modulo $q$.\\

In the last section, we will see that the proof of Theorem \ref{theo-moment} can be adapted to deal with the case of Fourier coefficients of integral weight modular forms in arithmetic progressions of modulus $q=p^N$ with $p$ an odd prime and $N>1$. We will actually give a detailed sketch of the proof of the following theorem.

\begin{theo} 
\label{theo-moment2intro}
Let $f$ be a Hecke eigenform of level 1 and of even weight $\kappa$. Define $E(X,q,a)$ as in \eqref{defE}.
Let $\nu$ be a positive integer and $e\in\{\pm1\}$. There exists an explicit constant $C_\nu$ depending only
on $\nu$ such that for any odd prime power $q=p^N$ with $N\ge2$, any $X>0$ and any $\epsilon >0$ with 
$$1 \le Y^{(1+\epsilon)2^{\nu-2}}<C_\nu q$$
 where $Y=q^2/X$, we have
 \begin{align}
 \frac{2}{\varphi(q)}\!\!\!\sum_{\tiny\begin{array}{c}a\:[q]\\ \legendre{a}=e\end{array}}\!\!\!E(X,q,a)^\nu &=
   \delta_{2\mid\nu}\frac{\nu !}{(\nu/2)!}\left(\frac{1}{Y}\sum_{\tiny{\begin{array}{c}1\le m <Y^{1+\epsilon}\\ \left(\frac{m}{p}\right)=e\end{array}}}
 \hat f_\infty(m)^2B^2\left(\frac{m}{Y}\right)\right)^{\nu/2}\\
 &\phantom{aaaaaaaaaaaaaaaaaaaaaaaaaaaa} + O\left(Y^{-1/2 + \epsilon}+\frac{Y^{\nu/2+\epsilon}}{p}\right)
 \nonumber
 \end{align}
where $B$ is a smooth rapidly decreasing function depending only on $w$ and $k$. 
\end{theo}

As before, we will deduce the following corollary.

 \begin{cor}
\label{cor-horizontal2}
 Let $f$ be a Hecke eigenform of level 1 and of even weight $\kappa$. Assume that its Fourier coefficients are real. Define $E(X,q,a)$ as in \eqref{defE}. Let $(q_k)_{k\ge1}$ be any sequence of odd prime powers, say $q_k=p_k^{N_k}$, with $(p_k)_k$ the sequence of odd
 prime numbers and $N_k\ge2$. If $X_k$ is a function of $k$ satisfying
 \begin{equation}
 \label{conditions2}
 \frac{q_k^{2}}{X_k}=o(\log\log p_k) \text{\; and \;} X_k=o(q_k^2) \text{\;\;\;as $k\to+\infty$}
 \end{equation} 
 then there exists a subsequence of the random variables
 $$\begin{array}{ccc}
 \qfrac{\setZ}{q_k\setZ}^\times &\to&\setR\\
 a&\mapsto& E(X_k,q_k,a)
 \end{array}$$ 
 which converges in law to the mixed distribution
 $$\frac{1}{2}\delta_0 + \frac{1}{2}\N(0,2V_{f,w})$$
 where $\delta_0$ is the Dirac measure at $0$ and $\N(0,2V_{f,w})$ is a Gaussian distribution of mean $0$ and variance $2V_{f,w}=2\frac{(4\pi)^{\kappa}}{\Gamma(\kappa)}\langle f, f\rangle ||w||_2^2$ with $\langle f, f\rangle $ the 
 square of the normalized Petersson norm of $f$.
 \end{cor}

\vskip 0.5 cm \subsection{Overview of the proof of Theorem \ref{theo-moment}}

Classically, we first show by detecting the congruence and applying a Vorono\u\i\; formula, that $E(X,q,a)$ is roughly equal to
\begin{equation}
\label{tordu}
\frac{1}{\sqrt Y}\sum_{ m\ge 1}\hat f_0(m){\rm Sa}_{q}(ma)B\left(\frac{m}{Y}\right)
\end{equation} 
where $Y=4q^2/X$, $B$ is a smooth rapidly decreasing function depending on $f$ and $w$, $\hat f_0(n)$ are the  normalized Fourier coefficients at the cusp 0 and 
$$\text{Sa}_q(x) =\left\{\begin{array}{cl} 
2 \cos\left(\frac{2\pi y}{q}\right) & \text{ if $\legendre{x}=1$ and where $y^2=x\:[q]$}\\
0 & \text{ otherwise}
\end{array}\right.$$
which is essentially the value of a normalized Salié sum as we will see in section 4.\\

Therefore, computing the moments of $E(X,q,\cdot)$ involves the computation of moments of Salié sums and here lies the main difference with the proofs of Theorem \ref{theo-FGKM} and \ref{theo-KR}. Indeed, in \cite{FGKM} and \cite{KR} the Salié sums are replaced by Kloosterman sums. The authors then appeal to the theory of trace functions and to the fact that the monodromy 
groups attached to the products of Kloosterman sums are pairwise independent and independent of $p$. 

This no longer holds in our case since on one hand, $q$ is not necessarly a prime $p$ and on the other hand, even if one assumes $q=p$, the monodromy group of a Salié sum is a dihedral group of order $2p$ (which depends on $p$) and there is no independence of these groups when multiplicative shifts occur. Thus, as we will see, the moments of Salié sums cannot 
converge.  Yet we will get around this problem.\\

\vskip 0.5 cm \subsection{Notations}
In the whole paper, we will use $a\:[q]$ for $a$ modulo $q$ and $e_q(x)=e(\frac{2i\pi x}{q})=e^{\frac{2i\pi x}{q}}$ for any real number $x$.  For integers $a$ and $b$, put $\llbracket a, b\rrbracket=[a,b]\cap\setZ$.
If $u$ is an integer coprime to $q$ then $\bar u$ is an integer such that $u\bar u=1[q]$ and if $q$ is an odd prime power and $u$ is an invertible square modulo $q$ then we denote by $\sqrt{u}^q$ its squareroot modulo $q$ in $\llbracket1,(q-1)/2\rrbracket$.

The symbols ${\sum}^\times$ and ${\sum\:}^\Box$ mean that we are restricting the summation respectively to invertible classes and invertible squares modulo $q$.

If $\M$ is a function defined on $\setZ/q\setZ^\times$ then we will denote by $\setE, \setE^+$ and $\setE^-$  the expected values respectively over the invertible classes, the invertible squares and invertible non-squares  modulo $q$. Precisely
$$\setE(\M)=\frac{1}{\varphi(q)}{\sum_{a\:[q]}}^\times \M(a)\text{ \;\;and \;\;} \setE^\pm(\M)=\frac{2}{\varphi(q)}\underset{\tiny\begin{array}{c} a\:[q]\\ \legendre{a}=\pm1\end{array}}{{\sum}^\times} \M(a)$$
where $\varphi$ is  Euler's totient function.

Let $\delta_q$ be the Dirac function at $0$ modulo $q$ \text{i.e.} for any integer $x$, one has 
$$\delta_q(x)=\left\{\begin{array}{cc}1 & \text{ if } x=0\:[q]\\ 0 & \text{otherwise.}\end{array}\right.$$
 We also write $\delta_0$ for the classical Dirac function at $0$.

For brevity, we will not keep track of the dependency on $X$ in the sums $S$ and $E$ and take the following new notations.
\begin{align}
\label{newS}
S_q := S(X,q,\cdot),\\
E_q := E(X,q,\cdot).
\label{newE}
\end{align} 

Classically, we let $$\calH=\{z\in\setC \:|\: {\rm Im\:} z > 0\}$$ and
$$\Gamma_0(N)=\left\{\Matrix{a&b\\c&d}\in\text{SL}_2(\setZ) \;\;\vline\;\; c=0\:[N]\right\}$$
for any positive integer $N$.\\

Finally, for any finite set $A$ we denote by $|A|$ its cardinality.

\vskip 0.5 cm \subsection{Acknowledgements}

The author would like to express his deepest thanks and appreciation to Florent Jouve and Guillaume Ricotta for their many words of advice which have made this paper a lot better. Also, the author would like to thank the anonymous referees for suggesting relevant  
modifications and helpful comments on a preliminary version of this work.

\vskip 1cm\section{Quick review of modular forms of half integral weight}

We recall some basic facts about the half-integral weight case, see for example \cite{Ono} for more details. 

By convention we denote by $\sqrt z$ the square root of the complex number $z$ with argument in $(-\pi/2,\pi/2]$ and  $\left(\frac{c}{d}\right)$ the Kronecker symbol for any interger $c$ and $d$ (see \cite[page 11]{Ono}) with $\left(\frac{\cdot}{2}\right)$ being the principal character modulo 2. We also define for any odd integer $d$
$$\epsilon_d = \left\{\begin{array}{cc}1 &  \text{ if } d=1\:[4], \\ i &  \text{ if } d=3\:[4]. \end{array}\right.$$

Let $\ell$ be an integer. A modular form of weight $\ell+1/2$ for $ \Gamma_0(4)$ is a holomorphic function $f : \calH \to \setC$ defined on the upper half-plane such that : 
\begin{enumerate}
\item One has $$f(\gamma z) = \epsilon_d^{-(2\ell+1)}\left(\frac{c}{d}\right)^{2\ell+1}(cz+d)^{\ell+1/2}f(z)$$
for any $\gamma = \Matrix{a&b\\c&d}\in \Gamma_0(4)$ which acts on $z\in\calH$ by $\gamma z = \frac{az+b}{cz+d}$.
\item The function $f$ is holomorphic at every cusp which here means that for any $z\in\calH$
\begin{align}
 &f_\infty(z) := f(z) = \sum_{n\ge 0} \hat f_\infty(n)n^{\ell/2-1/4}e(nz)\\
&f_0(z) := (-2iz)^{-(\ell+1/2)}f(-\frac{1}{4z}) =  \sum_{n\ge 0}\hat f_0(n)n^{\ell/2-1/4}e(nz)\\
&f_{-\frac{1}{2}}(z) := (-4z+\frac{1}{2})^{-(\ell+1/2)}f\left(\frac{2z}{-4z+\frac{1}{2}}\right) =  \sum_{n\ge 0} \hat f_{-\frac{1}{2}}(n)n^{\ell/2-1/4}e(nz).
\end{align}
\end{enumerate} 

So the $\hat f_\mathfrak a(n)$ are the normalized Fourier coefficients of $f$ at each of the three inequivalent cusps $\mathfrak a\in\{\infty, 0, -\frac{1}{2}\}$. If moreover,
\begin{enumerate}[resume]
\item $\hat f_\mathfrak a(0)=0$ for all $\mathfrak a\in\{\infty, 0, -\frac{1}{2}\}$ 
\end{enumerate} 
then $f$ is said to be cuspidal and the space of such forms is denoted by $\calS_{\ell+1/2}$. Since $\calS_{\ell+1/2}=\{0\}$ for $\ell\le 3$ we may suppose in the sequel that $\ell \ge 4$. 

\begin{rmq}\
\vskip 0.3cm
\begin{itemize}
\item[$\bullet$] Conditions (2.1), (2.2) and (2.3) may not be standard conventions but it appears there are several of them. See for example the definitions given in \cite{Hulse}, \cite{sign} or \cite{sign2}. 
\item[$\bullet$] The function $f_0$ is the image of $f$ by the Fricke involution for the group $\Gamma_0(4)$. Note also that $f_0\in\calS_{\ell+1/2}$ and if the coefficients $\hat f_\infty(n)$ are all real then so are the $\hat f_0(n)$.  
\end{itemize}
\end{rmq}

The Hecke operators on $\calS_{\ell+1/2}$  are non-zero only for square integers. If   $f\in\calS_{\ell+1/2}$ then
$$T_{p^2}f(z) = \sum_{n\ge 1}(p^{\ell-1/2}\hat f_\infty(p^2n)+ \chi_n(p)p^{\ell-1}\hat f_\infty(n) +\left(\frac{p}{2}\right)p^{\ell-1/2}\hat f_\infty(n/p^2))n^{\ell/2-1/4}e(nz)$$
with $p$ a prime number and $$\chi_t(d) =\left(\frac{d}{2}\right) \left(\frac{(-1)^\ell t}{d}\right).$$
 By convention, $\hat f_\infty(n/p^2)=0$ if $p^2\nmid n$. Also,
$$T_{m^2n^2}=T_{m^2}T_{n^2} $$ for any coprime odd integers $m$ and $n$.

We say that $f$ is an eigenform if it is an eigenvector for \emph{all} the Hecke operators. 

It follows from the Waldspurger formula (\cite{Wald}) and the bound for central values of 
automorphic $L$-functions (\cite{conrey}) that for any $f\in\calS_{\ell+1/2}$ and any squarefree integer $t$,
\begin{equation}
\label{boundsqfr}
|\hat f_\infty(t)|\ll_{f,\epsilon} t^{\alpha +\epsilon}
\end{equation}
with $\alpha=1/6$ and conjecturally $\alpha=0$ (see the begining of \cite{conrey}). 

If $f$ is an eigenform then it follows from \cite[Lemma 3.3]{squarefree} that its Fourier coefficients satisfy   
\begin{equation}
\label{bound}
|\hat f_\mathfrak{a}(n)|\ll_{f,\epsilon} n^{\alpha +\epsilon}
\end{equation}
for any $\epsilon>0$, any $\mathfrak{a}\in\{\infty,0,-\frac{1}{2}\}$ and any positive integer $n$.

\vskip 1cm
In order to apply a Vorono\u\i\;type formula, we will need a functional equation for the $L$-function of $f$ twisted by an additive character.  Such an equation is obtained in \cite[Lemma 4.3]{Hulse}. Let us recall first the classical functional equation for the $L$-function attached to $f$. For any $s\in\setC$,

$$\Lambda(s,f) := 2^s\Gamma_\ell(s)\sum_{n\ge1}\hat f_\infty(n)n^{-s} = \Lambda(1-s,f_0) $$
where $\Gamma_\ell(s):=(2\pi)^{-(s+\ell/2-1/4)}\Gamma(s+\ell/2-1/4)$.

\begin{lemme}[\cite{Hulse}]
\label{lemme-eqfonc}
Let $f\in\calS_{\ell+1/2}$. Let $u$ and $q$ be two integers with $q$ odd and $(u,q)=1$. Put $$L(s,f,u/q)=\sum_{n\ge1}\hat f_\infty(n)e_q(un)n^{-s}$$
then $L(s,f,u/q)$ is absolutely convergent for {\rm Re} $s>1$ and extends to an entire function sastifying the functional equation 
$$(2q)^s\Gamma_\ell(s)L(s,f,u/q)=\epsilon_q^{-(2\ell+1)}\left(\frac{-\bar u}{q}\right)(2q)^{1-s}\Gamma_\ell(1-s) L(1-s,f_0,-\overline{4u}/q)$$ 
with $u\bar u = 1 \:[q]$. 

Moreover, this $L$-function is of polynomial growth on vertical strips \text{i.e.} for any $\sigma_1<\sigma_2$, there exists $A>0$ such that
$$|L(s,f,u/q)|\ll_{f, q, \sigma_1,\sigma_2} (1+|t|)^{A}$$ for all $s=\sigma+it$ with $\sigma\in[\sigma_1,\sigma_2]$ and $t$ real. 
\end{lemme}

We finish this section by giving the fundamental property of the Rankin-Selberg $L$-function attached to $f$, see \cite[Proposition 7]{sign}. Here the assumption that $f$ is an eigenform is unnecessary.

\begin{prop}
\label{prop-RS}
Let $f,g\in\calS_{\ell+1/2}$, $s\in\setC$ and let $$D(s,f\times\bar g)=\sum_{n\ge1}\hat f_\infty (n)\overline{\hat g_\infty(n)} n^{-s}.$$

The classical Eisenstein series defined for $z\in\calH$ by
$$E_\infty(z,s)=\sum_{\gamma\in\Gamma_\infty\backslash \Gamma_0(4)}({\rm Im}\: \gamma z)^s$$
converges for {\rm Re} $s>1$ and extends in the $s$ variable to a meromorphic function. In the half-plane {\rm Re} $s\ge1/2$, there is only
one simple pole at $s=1$ with residue $$\underset{s=1}{\rm{Res\:}} E_\infty(z,s)= \rm{Vol}(\Gamma_0(4)\backslash\calH)^{-1}.$$

Moreover, $D(s,f\times\bar g)$ converges absolutely for {\rm Re}$\:s> 1$ and we have 
$$(4\pi)^{-(s+\ell-1/2)}\Gamma(s+\ell-1/2)D(s,f\times\bar g)=\int_{\Gamma_0(4)\backslash\calH}f(z)\overline{g(z)}y^{\ell+1/2}E_\infty(z,s)\frac{dxdy}{y^2}$$
so  $D(s,f\times\bar g)$ extends to a meromorphic function on the region {\rm Re} $s\ge1/2$ with at most one pole, located at $s=1$. Also, one has
$$(4\pi)^{-(\ell+1/2)}\Gamma(\ell+1/2)\underset{s=1}{\rm{Res\:}} D(s,f\times\bar g)= \langle f, g\rangle$$
with
$$\langle f, g\rangle :=\frac{1}{\rm{Vol}(\Gamma_0(4)\backslash\calH)}\int_{\Gamma_0(4)\backslash\calH}f(z)\overline{g(z)}y^{\ell+1/2}\frac{dxdy}{y^2}.$$ 
\end{prop} 

We will also need the two following classical results which we state in one lemma. 

\begin{lemme}
\label{lemme-RS}
Let $f\in\calS_{\ell+1/2}$ and $Y>0$. There exists a constant $c_f$ such that
$$\sum_{m\le Y}|\hat f_\infty(m)|^2 \sim c_f Y$$
 as  $Y\to+\infty$. Moreover, we have
 $$\sum_{m\le Y}\frac{|\hat f_\infty(m)|^2}{\sqrt m}\ll_f  \sqrt Y.$$
\end{lemme}

\begin{proof}
The first assertion is proved applying Perron's formula from Proposition \ref{prop-RS} (see \cite[Lemma 4.1]{sign2} and note that $f$ does not need to be an eigenform).

 The second assertion comes from the first one by a summation by parts.
\end{proof}

\begin{rmq}
Since $f_0\in \calS_{\ell+1/2}$, Lemma \ref{lemme-RS} can be stated for $f_0$ as well. 
\end{rmq}

\vskip 1 cm \section{The Vorono\u\i\;summation formula}

Classically, in order to study Fourier coefficients in arithmetic progressions, we use a generalization of the Poisson summation formula which involves the Mellin transform of the test function. This formula is named the Vorono\u\i\;summation formula (see \cite{voro} for more details).

\begin{prop}[Vorono\u\i\: formula]
Let $f\in\calS_{\ell+1/2}$. Let $u$ and $q$ be two integers with $q$ odd and $(u,q)=1$. Let $X>0$ and $w$ be a smooth non-zero real-valued function with compact support in $]0, +\infty[$. Then 
\begin{equation}
\label{voro}
\sum_{n\ge1}\hat f_\infty(n)e_q(un)w\left(\frac{n}{X}\right) = \epsilon_q^{-(2\ell+1)}\left(\frac{-\bar u}{q}\right)\frac{X}{2q}\sum_{m\ge1}\hat f_0(m)e_q(-\overline{4u}m)B\left(\frac{m}{4q^2/X}\right)
\end{equation}
with 
\begin{equation}
\label{bessel}
B(x)=\frac{1}{2i\pi}\int_{(\sigma)}\frac{\Gamma_\ell(s)}{\Gamma_\ell(1-s)}\hat w(1-s)x^{-s}ds
\end{equation}
for $x>0$, $\sigma>1$ and $\hat w(s)=\int_0^{+\infty}w(x)x^s\frac{dx}{x}$.
\end{prop}

\begin{proof}
Using the inverse Mellin transform for $w$ we get
$$A:=\sum_{n\ge1}\hat f_\infty(n)e_q(un)w\left(\frac{n}{X}\right) = \frac{1}{2i\pi}\int_{(\sigma)}L(s,f,u/q)X^s\hat w(s)ds$$ 
for $\sigma > 1$. The smoothness of $w$ implies that $\hat w$ is rapidly decreasing on vertical strips and since the $L$-function is of polynomial growth, we can shift the domain of integration to Re $s=1-\sigma$ and then, using the functional equation, we get

$$A=\frac{\omega}{2i\pi}\int_{(1-\sigma)}(2q)^{1-2s}\frac{\Gamma_\ell(1-s)}{\Gamma_\ell(s)} L(1-s,f_0,-\overline{4u}/q)X^s
\hat w(s)ds$$
where $\omega= \epsilon_q^{-(2\ell+1)}\left(\frac{-\bar u}{q}\right)$. Changing $s\mapsto1-s$, we have
$$A= \frac{\omega}{2i\pi}\int_{(\sigma)}(2q)^{-1+2s}\frac{\Gamma_\ell(s)}{\Gamma_\ell(1-s)} L(s,f_0,-\overline{4u}/q)X^{1-s}
\hat w(1-s)ds$$
which is enough to conclude since the $L$-function is absolutely convergent for Re $s>1$.

\end{proof}

\begin{rmq}
Of course, $B$ depends on $w$ and on the weight of $f$. For brevity we will not keep track of this dependency.
\end{rmq}

Let us  prove some useful properties on the function $B$.

\begin{lemme}
\label{lemme-B}
The function $B$ is smooth and real on $]0,+\infty[$. 

If $x> 0$ then $$|B(x)| \ll_{A,\ell,w} x^{-A}$$
for any $A>0$. 

If $0<x\le 1$ then  $$|B(x)| \ll_{\epsilon,w} x^{-(1/4-\ell/2+\epsilon)}$$ 
for any $\epsilon>0$. 

In particular, $$|B(x)|\ll_{\ell,w,\epsilon}1$$ for $x>0$ and the Mellin transform of $B$ is well defined for {\rm Re} $s > -(\ell/2-1/4)$ and rapidly decreasing on vertical strips.
\end{lemme}

\begin{proof}Let $x>0$.
It is clear that $\overline{B(x)} = B(x)$. The smoothness of $w$ implies that $\hat w$ is rapidly decreasing on vertical strips and shifting the domain of integration from $\sigma>1$ to any $A>0$, we get
$$|B(x)|\le \frac{1}{2\pi}\int_{-\infty}^{+\infty}\frac{|\Gamma_\ell(A+it)|}{|\Gamma_\ell(1-A-it)|(1+|t|)^r}x^{-A}dt$$ for any $r>0$. Applying the Stirling formula to $\Gamma_\ell$, we obtain the first inequality stated in the lemma.

 The same trick applies to obtain the second statement, recalling that 
 $\frac{\Gamma_\ell(s)}{\Gamma_\ell(1-s)}$ has no pole on Re $s>-(\ell/2-1/4)$ so we can shift the domain 
 of integration to $\sigma=-(\ell/2-1/4)+\epsilon$ for any $\epsilon>0$.
 
\end{proof}

From this we deduce the following.
\begin{prop}
\label{prop-B}
Let $\eta>0$. If $Y\ge1$ then 
$$\sum_{m\ge Y^{1+\eta}}|\hat f_0(m)B\left(\frac{m}{Y}\right)| \ll_{\eta, A} Y^{-A}$$ for any $A\ge1$.
Moreover, if $Y>0$ then
$$\sum_{m\ge 1}|\hat f_0(m)B\left(\frac{m}{Y}\right)| \ll_\epsilon Y^{1+\epsilon}$$ for any $\epsilon>0$.
The implicit constants depending also on $f$ and $w$.
\end{prop} 

\begin{proof}
Let $Z\ge1$. Applying the Cauchy-Schwarz inequality and using Proposition \ref{prop-RS} for $f=g=f_0$ combined with Lemma 
\ref{lemme-B}, we get
\begin{align*}
\left(\sum_{m\ge Z} |\hat f_0(m)B\left(\frac{m}{Y}\right)|\right)^2 
&\le \sum_{m\ge Z} |\hat f_0(m)|^2m^{-(1+\epsilon)} \sum_{m\ge Z}B^2\left(\frac{m}{Y}\right)m^{1+\epsilon}\\ 
&\ll_{A,\epsilon} Y^A \sum_{m\ge Z}\frac{1}{m^{A-1-\epsilon}}\\
&\ll_{A,\epsilon} \frac{Y^A}{Z^{A-2-2\epsilon}}.
\end{align*}
Taking $\epsilon$ fixed, $Z=Y^{1+\eta}$  and $A$ sufficiently large we obtain the first bound. Take $Z=1$ to have the second bound when $0<Y<1$ (actually a better one that we will not need). 

Otherwise if $Y\ge1$, take $Z=Y$, then
\begin{align*}
\sum_{m\ge Y} |\hat f_0(m)B\left(\frac{m}{Y}\right)|
&\ll_{\epsilon} Y^{1+\epsilon}
\end{align*}
but also
\begin{align*}
\sum_{m\le Y} |\hat f_0(m)B\left(\frac{m}{Y}\right)|
&\ll \left(\sum_{m\le Y} |\hat f_0(m)|^2\right)^{1/2} \left(\sum_{m\le Y}\left(\frac{m}{Y}\right)^{\ell-1/2-\epsilon}\right)^{1/2} & \text{\;\;\; from Lemma \ref{lemme-B}}\\
&\ll Y^{1/2}\cdot Y^{1/2}
\end{align*}
using Lemma \ref{lemme-RS}.

\end{proof}

We finish this section by stating the following Plancherel formula.

\begin{lemme}
\label{lemme-plancherel}
One has
\begin{equation}
||w||_2=||B||_2
\end{equation}
where $||\cdot||_2$ is the $L^2$ norm on $]0,+\infty[$.
\end{lemme}

\begin{proof}
First note that $B$ is the inverse Mellin transform of $s\mapsto \frac{\Gamma_\ell(s)}{\Gamma_\ell(1-s)}\hat w(1-s)$. Hence for $\sigma >1$

\begin{align*}
\int_0^{+\infty} B(x)^2dx &= \frac{1}{2i\pi}\int_{(\sigma)}\frac{\Gamma_\ell(s)}{\Gamma_\ell(1-s)}\hat w(1-s)\int_0^{+\infty} B(x) x^{1-s} \frac{dx}{x}ds\\
&= \frac{1}{2i\pi}\int_{(\sigma)}\frac{\Gamma_\ell(s)}{\Gamma_\ell(1-s)}\hat w(1-s)\frac{\Gamma_\ell(1-s)}{\Gamma_\ell(s)}\hat w(s)ds\\
&=  \frac{1}{2i\pi}\int_{(\sigma)}\int_0^{+\infty} w(x)x^{1-s}\hat w(s) \frac{dx}{x}ds\\
&= \int_0^{+\infty} w(x)^2 dx.
\end{align*}

\end{proof}

\vskip 1 cm

 \section{Moments of Kloosterman-Salié sums}
In order to compute their moments, we give the exact formula for Kloosterman-Salié sums with prime power moduli. Proofs of these facts can be found in \cite[Lemma 8.4.3]{coursKo} or \cite[Lemmas 12.2, 12.3 and 12.4]{IwKo} and are based on the stationnary phase method. All the results of this section will be used in the next one. \\

Let $p$ be an odd prime. Let $N$, $m$, $n$ be integers with $N\ge1$ and let $q=p^N$.   We define the normalized Kloosterman and Salié sums respectively by 
\begin{align*}
{\rm Kl}_q(m,n)= \frac{1}{\sqrt q}{\sum_{x\: [q]}}^{\times}e_q(mx+n\bar x)\\
{\rm Sal}_q(m,n)= \frac{1}{\sqrt q}{\sum_{x\: [q]}}^{\times}\left(\frac{x}{q}\right)e_q(mx+n\bar x)
\end{align*}
where the sums are taken over the invertible classes modulo $q$ and $\bar x$ denotes the inverse of $x$ modulo $q$. 

\begin{prop}
\label{prop-Salie}
If $m$ and $n$ are coprime to $p$ then
$${\rm Sal}_q(m,n) = \left(\frac{m}{q}\right)\epsilon_q\sum_{x^2=mn \:[q]}e_q(2x)$$ and if $N\ge2$,
$${\rm Kl}_q(m,n) = \epsilon_q\sum_{x^2=mn \:[q]}\left(\frac{x}{q}\right)e_q(2x).$$

If $m$ is coprime to $p$ and $p\mid n$ then the sums vanish. Of course when $N$ is even these sums coincide.
\end{prop}

Note that if $m$ and $n$ are integers such that $\left(\frac{m}{p}\right)=\left(\frac{n}{p}\right)=1$ then 
$${\rm Sal}_q(m,n) = 2\epsilon_q \cos\left(\frac{4\pi{\sqrt{mn}}^q}{q}\right)$$ where ${\sqrt{x}}^q$ is the square root of $x$ mod $q$ taken as an integer in $\llbracket1,(q-1)/2\rrbracket$. It will be clear that in the sequel, all the results involving this notation will be in fact independent of the choice of the square root. Let us recall the following notation we introduced in section \ref{section1}. For any $x$ modulo $q$,
$$\text{Sa}_q(x) =\left\{\begin{array}{cc} 
2 \cos\left(\frac{2\pi{\sqrt{x}}^q}{q}\right) & \text{ if $\legendre{x}=1$}\\
0 & \text{ otherwise}
\end{array}\right.$$
We now compute the moments of such quantities.

\begin{lemme}
\label{lemme-momentsalie}
Let $p$, $q$ be as above and let $m_1,\ldots,m_\nu$ be some positive integers with $\nu\ge1$ and assume that $\legendre{m_i}=\legendre{m_1}=\pm 1$ for all $2\le i\le\nu$. Then
 $$\setE^\mp\left( \prod_{i=1}^\nu{\rm Sa}_q(m_ia)\right)=0$$ and
$$\setE^\pm\left( \prod_{i=1}^\nu{\rm Sa}_q(m_ia)\right)=  \frac{q}{\varphi(q)}\sum\limits_{\text{\boldmath $e$}\in\{\pm 1\}^\nu}\left(\delta_q\left({\sum\limits_{i=1}^\nu e_i{\sqrt{\mu_p m_i}}^q}\right) -\frac{1}{p}\:\delta_{q/p}\left({\sum\limits_{i=1}^\nu e_i{\sqrt{\mu_pm_i}}^{q/p}}\right)\right)$$ 
where $\mu_p$ is any integer such that $\legendre{\mu_p}=\legendre{m_1}$ and where, by convention, $\delta_{q/p}=1$ if $q=p$.\\

Obviously, if  $\legendre{m_1}\neq\legendre{m_i}$ for some $i$ the left hand-side of the equality vanishes.
\end{lemme}

\begin{proof} First note that ${\rm Sa}_q(m_ia)=0$ if $\legendre{a}\neq\legendre{m_i}=\legendre{\mu_p}$ so we just have to show the second equality where the expected value can be viewed as half the sum  taken over the $a=\mu_p b^2$ for any $b\:[q]$.
\begin{align*}
\setE^\pm\left( \prod_{i=1}^\nu{\rm Sa}_q(m_ia)\right)
&= \frac{1}{\varphi(q)}{\sum_{b\:[q]}}^\times \prod_{i=1}^\nu{\rm Sa}_q(m_i\mu_pb^2)\\
&=  \frac{1}{\varphi(q)}{\sum_{b\:[q]}}^\times\sum_{\text{\boldmath $e$}\in\{\pm 1\}^\nu}e_q\left(b\sum_{i=1}^\nu e_i{\sqrt{\mu_p m_i}}^q \right). 
\end{align*}
Using the explicit formula for the Ramanujan sum (\cite[page 50]{coursKo}), we get the conclusion.
\end{proof}

As in \cite[section 4]{KR} we would like to write this moment as a main term (independent of $m_i$ and $q$) plus an error term when $q\to +\infty$. However this seems to be impossible because as random variables, the $(\text{Sa}_q(m_i \:\cdot))_i$ cannot behave independently when $q\to +\infty$.  Indeed, we have 
$$\text{Sa}_q(m x)= 2T_{{\sqrt m}^q}(\text{Sa}_q(x)/2)$$
with $\legendre{m}=\legendre{x}=1$ and where $T_n$ is the Tchebychev polynomial of degree $n$.\\

Nevertheless, we manage to give an approximation of the Salié moment whose proof involves the following definition.

\begin{defi}
Let $r\ge2$ and $\E_1$ the set of all $\text{\boldmath $e$}\in\{\pm 1\}^r$ such that $e_1=1$. We define the $r$ variables polynomial $Q_r$ with coefficients in $\setZ$ by 
$$Q_r(x_1^2,\ldots,x_r^2)=\prod_{\text{\boldmath $e$}\in \E_1}\sum_{i=1}^r e_ix_i.$$

It is well defined since the polynomial on the right-hand side is an even function in each variable $x_i$. Note that $\deg Q_r=2^{r-2}$. For $r\in\{0,1\}$, we put $Q_1(x)=x$ and $Q_0 = 0$.
\end{defi}

\begin{lemme}
\label{lemme-moment}
Let $Y>0$ such that $(\nu^2 Y)^{2^{\nu-2}}<q/2$. Then for any natural numbers $m_1,\ldots,m_\nu<Y$ with $\legendre{m_i}=1$ for all $i$, we have

$$\sum\limits_{\text{\boldmath $e$}\in\{\pm 1\}^\nu}\delta_{q}\left(\sum\limits_{i=1}^\nu e_i{\sqrt m_i}^q\right) \le 2^\nu 
\sum\limits_{\text{\boldmath $e$}\in\{\pm 1\}^\nu}\delta_0\left(\sum\limits_{i=1}^\nu e_i\sqrt m_i \right)$$
where  $\sqrt m_i$ is the non negative square root of $m_i$ in $\setZ$.
\end{lemme}

\begin{proof}
If $\sum\limits_{i=1}^\nu e_i{\sqrt m_i}^q = 0\:[q]$ then $Q_\nu(m_1,\cdots,m_\nu)= 0 \:[q]$. Thus,
$$\sum\limits_{\text{\boldmath $e$}\in\{\pm 1\}^\nu}\delta_{q}\left(\sum\limits_{i=1}^\nu e_i{\sqrt m_i}^q\right)\le 2^\nu \delta_{q}\left(Q_\nu(m_1,\cdots,m_\nu)\right). $$

The degree and coefficients of the polynomials $Q_\nu$ do not depend on $q$, also we have
$$|Q_{\nu}(m_1,\cdots,m_\nu)|=\prod_{\text{\boldmath $e$}\in \E_1}|\sum_{i=1}^\nu e_i\sqrt m_i|\le (\nu \sqrt Y)^{2^{\nu-1}} < q/2$$
 so we can replace $\delta_q$ by $\delta_0$ in the right hand-side of the previous inequality. However, if $Q_\nu(m_1,\cdots,m_\nu)= 0$ then there exists $\text{\boldmath $e$}\in\{\pm 1\}^\nu$ such that  $\sum\limits_{i=1}^\nu e_i\sqrt m_i = 0$ and the result follows.
\end{proof}

\begin{rmq}
Surprisingly, the use of the polynomials $Q_\nu$ to deal with the quantities $\sum\limits_{i=1}^\nu e_i{\sqrt m_i}$ is done in a very different context in \cite[Lemma 3.4]{HughesRud}.
\end{rmq}

\vskip 1 cm \section{Asymptotic evaluation of the moments}

In this section, we prove a strong form of Theorem \ref{theo-moment}. Let $q=p^N$ be a power of an odd prime and for $a$ coprime to $q$, define $S_q(a)$ and $E_q(a)$ as in \eqref{newS}, \eqref{newE} and \eqref{defS}, \eqref{defE}   for a fixed modular form $f$ satisfying the conditions of Theorem \ref{theo-moment}. We will compute the moments of $E_q$ using the previous results of sections 2, 3 and 4. First, we will apply the Vorono\u\i\: summation formula to make the link with Salié sums. Then we will use the formula that we stated in the previous section for the moments of Salié sums (Lemma \ref{lemme-momentsalie}) and we will apply the approximation formula that we have just proved (Lemma \ref{lemme-moment}).

\vskip 0.5 cm \subsection{Applying the Vorono\u\i\: summation formula}
\begin{prop} 
\label{prop-moment}
Let $\nu$ a positive integer and put $Y=4q^2/X$. One has for any $\eta>0$, $A>0$, $\epsilon>0$ and $a$ coprime to $q$ 
\begin{equation}
\label{moment}
E_q(a)^\nu=\left(\frac{\epsilon_{q}^{-2\ell}}{\sqrt Y}\sum_{1\le m< Y^{1+\eta}}\hat f_0(m){\rm Sa}_{q}(ma)B\left(\frac{m}{Y}\right)\right)^\nu + O\left(Y^{\nu/2+\epsilon}p^{-3/2}+Y^{-A}\right),
\end{equation}
the implicit constant depending on $f, w, \eta, \epsilon$ and $A$.
\end{prop}

\begin{proof}

We first use additive characters to write

\begin{align*}
S_q(a)&=\frac{1}{q}\sum_{b\:[q]}e_q(-ab)\sum_{n\ge1}\hat f_\infty(n)e_q(nb)w\left(\frac{n}{X}\right)\\
&=\frac{1}{q}{\sum_{b\:[q]}}^\times e_q(-ab)\sum_{n\ge1}\hat f_\infty(n)e_q(nb)w\left(\frac{n}{X}\right)\\
&\phantom{xxxxxxxxxxxxxxxx}+\;\;\;\underbrace{\frac{1}{q}\sum_{r=0}^{N-1}{\sum_{b\:[p^r]}}^\times e_{p^r}(-ab)\sum_{n\ge1}\hat f_\infty(n)e_{p^r}(nb)w\left(\frac{n}{X}\right)}_{=\:\frac{1}{p}S_{p^{N-1}}(a)}.
\end{align*}
We will show that the last term is in fact negligable.

Applying the Vorono\u\i\; formula, we have for any $r\in \llbracket1,N\rrbracket$ and $b$ coprime to $p^r$,
\begin{align*}
&\frac{1}{q}{\sum_{b\:[p^r]}}^\times e_{p^r}(-ab)\sum_{n\ge1}\hat f_\infty(n)e_{p^r}(nb)w\left(\frac{n}{X}\right)\\
&\phantom{xxxxxx}= \frac{\epsilon_{p^r}^{-(2\ell+1)}X}{2p^{N+r}}{\sum_{b\:[p^r]}}^\times e_{p^r}(-ab)\left(\frac{-\bar b}{p^r}\right)\sum_{m\ge1}\hat f_0(m)e_{p^r}(-\overline{4b}m)B\left(\frac{m}{4p^{2r}/X}\right)\\
&\phantom{xxxxxx}= \frac{\epsilon_{p^r}^{-(2\ell+1)}X}{2p^{N+r/2}}\sum_{m\ge1}\hat f_0(m)\text{Sal}_{p^r}(a,\bar 4m)B\left(\frac{m}{4p^{2r}/X}\right)\\
&\phantom{xxxxxx}=\sqrt{\frac{X}{q}}\:\frac{\epsilon_{p^r}^{-(2\ell+1)}p^{(N-r)/2}}{\sqrt Y}\sum_{m\ge1}\hat f_0(m)\text{Sal}_{p^r}(a,\bar 4m)B\left(\frac{mp^{2(N-r)}}{Y}\right).
\end{align*}

Proposition \ref{prop-B} shows that this last quantity is $\ll_\epsilon \sqrt{\frac{X}{q}}\: Y^{1/2+\epsilon}p^{-\frac{3}{2}(N-r)}$ and  the same holds for $r=0$ using directly the Mellin transform
so we get 

\begin{align*}
E_q(a)&=\frac{\epsilon_{q}^{-(2\ell+1)}}{\sqrt Y}\sum_{m\ge1}\hat f_0(m)\text{Sal}_{q}(a,\bar 4m)B\left(\frac{m}{Y}\right)+ O(Y^{1/2+\epsilon}p^{-3/2})\\
&= \frac{\epsilon_{q}^{-2\ell}}{\sqrt Y}\sum_{m\ge1}\hat f_0(m)\text{Sa}_{q}(ma)B\left(\frac{m}{Y}\right)+ O(Y^{1/2+\epsilon}p^{-3/2})\\
&=\frac{\epsilon_{q}^{-2\ell}}{\sqrt Y}\sum_{1\le m< Y^{1+\eta}}\hat f_0(m)\text{Sa}_{q}(ma)B\left(\frac{m}{Y}\right)+ O(Y^{1/2+\epsilon}p^{-3/2}+Y^{-A})
\end{align*}
using Proposition \ref{prop-B} again for any $\eta>0$ and $A\ge1$. The first sum is $\ll Y^{1/2+\epsilon}$ so raising $E_q$ to the  $\nu$-power, we obtain \eqref{moment}.

\end{proof}
\vskip 0.3cm
\begin{rmq}\
\begin{itemize}
\item[$\bullet$] The error term for $r=0$ in the proof is viewed in \cite{FGKM} and \cite{KR} as a fake main term of $S_q$ since it does not depend on $a$. But obviously, including this main term or not in the definition of $S_q$ does not change the result.

\item[$\bullet$] We also proved that 
\begin{equation}
\label{sommemoins}
\frac{S_{p^N}(a)-\frac{1}{p}S_{p^{N-1}}(a)}{\sqrt{X/p^N}} = \frac{\epsilon_{q}^{-2\ell}}{\sqrt Y}\sum_{1\le m< Y^{1+\eta}}\hat f_0(m)\text{Sa}_{q}(ma)B\left(\frac{m}{Y}\right)+ O(Y^{-A})
\end{equation}
which appears to be the ``right'' normalisation of $S_{p^N}$ if we seek for an asymptotic expansion when $p$ is fixed and $N\to+\infty$. 
\end{itemize}
\end{rmq}

\vskip 0.5 cm \subsection{Applying Lemma \ref{lemme-momentsalie}}

We want to give an asymptotic formula for the expected value of the main term in \eqref{moment} when $Y$ and $q$ tend to $+\infty$. Therefore, we define
\begin{align}
\label{sommetordu}
\M_q(a)&:=\frac{1}{\sqrt Y}\sum_{1\le m< Y^{1+\eta}}\hat f_0(m)\text{Sa}_{q}(ma)B\left(\frac{m}{Y}\right).
\end{align}

 First note that 
\begin{equation}
\label{plusmoins}
\setE(\M_q^\nu)= \frac{1}{2}\setE^+(\M_q^\nu)+\frac{1}{2}\setE^-(\M_q^\nu).
\end{equation}

Since the computation of $\setE^-$ is basically the same as for $\setE^+$, we focus on the latter.

\begin{align}
\setE^+(\M_q^\nu) &= \setE^+\left(\frac{1}{Y^{\nu/2}}\left(\sum_{1\le m< Y^{1+\eta}}\hat f_0(m)\text{Sa}(ma)
B\left(\frac{m}{Y}\right)\right)^{\nu}\right) \\
   &= \frac{1}{Y^{\nu/2}} \sum_{1\le m_1,\ldots ,m_\nu< Y^{1+\eta}}\prod_{i=1}^\nu\hat f_0(m_i)
   B\left(\frac{m_i}{Y}\right)\setE^+\left(\prod_{i=1}^\nu\text{Sa}(m_ia)\right)
   \label{EMnu}
\end{align}

 By Lemma \ref{lemme-momentsalie}, if  $\legendre{m_i}=1$ \text{for all}  $i$ then 
\begin{equation}
\label{Esalie}
\setE^+\left(\prod_{i=1}^\nu\text{Sa}(m_ia)\right) =(1-1/p)^{-1}
\sum\limits_{\text{\boldmath$e$}\in\{\pm 1\}^\nu}\left(\delta_{q}\left(\sum\limits_{i=1}^\nu e_i{\sqrt{ m_i}}^q\right) -\frac{1}{p}\:\delta_{q/p}\left(\sum\limits_{i=1}^\nu e_i{\sqrt{m_i}}^{q/p}\right)\right) 
\end{equation}
otherwise this quantity vanishes.\\

Hence, we look for a main term in 

\begin{equation}
\label{eq1}
 \underset{\tiny\begin{array}{c}1\le m_i < Y^{1+\eta}\\1\le i\le\nu\end{array}}{{\sum}^{\:\Box}}\prod_{i=1}^\nu\hat f_0(m_i)
B\left(\frac{m_i}{Y}\right)\sum\limits_{\text{\boldmath$e$}\in\{\pm 1\}^\nu}\delta_{q}\left(\sum\limits_{i=1}^\nu e_i{\sqrt m_i}^q \right)
\end{equation}
and recall that $\Box$ means we impose the condition  $\legendre{m_i}=1$ \text{for all}  $i$.\\

We will show  studying the squarefree part of the $m_i$ that the main term will come from the $\nu$-tuple $(m_1,\dots,m_\nu)$ such that $|\{m_i, i\in\llbracket1,\nu\rrbracket\}|=\nu/2$ for even $\nu$.  \\

We will first do an initial cleaning of \eqref{eq1} so we can apply the main result of section 4 to get precise formulas for odd and even moments.

\vskip 0.5 cm \subsection{Combinatorial aspect}
As in \cite[Lemma 7.1]{KR} we will rearrange the sum in \eqref{eq1} according to which squarefree parts of the $m_i$ appear. However, we use a different approach and notation.  For $s\in\llbracket1,\nu\rrbracket$, we denote by $S(\nu,s)$ the set of surjective functions from $\llbracket1,\nu\rrbracket$ to $\llbracket1,s\rrbracket$. For $j\in\llbracket1,s\rrbracket$, let
$$\sigma^{-1}(j)=\{i\in\llbracket1,\nu\rrbracket\:\vline\:\sigma(i)=j\}$$ and $|\sigma^{-1}(j)|$ its cardinal. For $k\in\llbracket1,\nu\rrbracket$, let
 $$\sigma_k =|\{j\in\llbracket1,s\rrbracket\: :\: |\sigma^{-1}(j)|=k\}|.$$ 
 
 If $m\in\setN^*$ then it can be uniquely written as $m=r^2t$ with $t$ squarefree and $r\ge1$. \underline{From now on}, if $\sigma\in S(\nu,s)$ then  for $1\le i\le \nu$ and $1\le j\le s$ the letters $t_j$, $r_i$ and $m_i$ will always refer to positive integers such that $t_j$ is squarefree and $1\le m_i<Y^{1+\eta}$. Also, the symbol $\ll$ will be used  in the sense that the implicit constant does not depend on $p$, $q$ or $Y$ but of course may depend on the other parameters.
 
 Before coming back to our problem, note that for any set of integer $M$, any complex-valued function $F$ defined on $M^\nu$, we
 have
 \begin{equation}
\label{reorg}
\sum_{\mathbf m\in M^\nu} F(m_1,\ldots,m_\nu)= \sum_{s=1}^\nu\sum_{\sigma\in S(\nu,s)}\sum_{t_1<\ldots<t_s}
\sum_{\tiny{\begin{array}{c}\mathbf m\in M^\nu\\m_i=r_i^2t_{\sigma(i)}\end{array}}}F(m_1,\ldots,m_\nu).
\end{equation}
 
 Therefore, by Equations \eqref{EMnu}, \eqref{Esalie} and \eqref{reorg} we state that

\begin{equation}
\label{eq}
\setE^+(\M_q^\nu)=\frac{(1-1/p)^{-1}}{Y^{\nu/2}}\sum_{s=1}^\nu\sum_{\sigma\in S(\nu,s)}\left(\Sigma_\sigma(q)-\frac{1}{p}\Sigma_\sigma(q/p)\right)
\end{equation}
where
$$\Sigma_\sigma(q)= \underset{\tiny{\begin{array}{c}t_1<\ldots<t_s\\m_i=r_i^2t_{\sigma(i)}\end{array}}}{{\sum}^{\:\Box}}\prod_{i=1}^\nu\hat f_0(m_i)
B\left(\frac{m_i}{Y}\right)\sum\limits_{\text{\boldmath $e$}\in\{\pm 1\}^\nu}\delta_{q}\left(\sum\limits_{i=1}^\nu e_i{\sqrt m_i}^q \right).$$
for any $s\in\llbracket1,\nu\rrbracket$ and $\sigma\in S(\nu,s)$.

\vskip 0.5 cm \subsection{Applying Lemma \ref{lemme-moment}}

We will show the following.

\begin{prop}
\label{prop-sigmasigma}
Let $C_\nu=\big(2\nu^{2^{\nu-1}}\big)^{-1}$. Assume  there exists $\delta>0$ such that $Y^{2^{\nu-2}+\delta}<C_\nu q$. Then for $\eta$ sufficiently small\footnote{We stress that we do not need to have $\eta\to0$. Here $\eta$ is a fixed constant (independent of $q$ and $Y$).} :
\begin{itemize}
\item [$\bullet$] If there exists $j\in\llbracket1,s\rrbracket$ such that $|\sigma^{-1}(j)|=1$ then $\Sigma_\sigma(q)=0$. 
\item [$\bullet$] Otherwise,
$$|\Sigma_\sigma(q)|\ll Y^{(1+\eta)(\frac{\nu}{2}-\frac{s-\sigma_2}{3}) + \epsilon }$$
 for any $\epsilon >0$. 
\end{itemize}
Hence, if $\sigma_2<s$ (\textit{e.g.} we are not in the case where $\nu$ is even and $s=\nu/2$), then 
$$\frac{1}{Y^{\nu/2}}|\Sigma_\sigma(q)|\ll Y^{-1/3+\epsilon}$$
for $\eta$ small enough so this term becomes a negligible contribution to $\setE(\M_q^\nu)$ when $Y\to+\infty$.\\
\end{prop}

\begin{proof}

 Take $\eta<\delta2^{2-\nu}$ so the condition of Lemma \ref{lemme-moment} holds for $Y^{1+\eta}$. By Lemma \ref{lemme-B} and \ref{lemme-moment}, we have
\begin{align*}
|\Sigma_\sigma(q)| &\ll \sum_{\tiny{\begin{array}{c}t_1<\ldots<t_s\\m_i=r_i^2t_{\sigma(i)}\end{array}}}
\prod_{i=1}^\nu|\hat f_0(m_i)|
\sum\limits_{\text{\boldmath $e$}\in\{\pm 1\}^\nu}\delta_0\left(\sum\limits_{i=1}^\nu e_i\sqrt m_i\right)\\
&= \sum\limits_{\text{\boldmath $e$}\in\{\pm 1\}^\nu} \sum_{\tiny{\begin{array}{c}t_1<\ldots<t_s \\m_i=r_i^2t_{\sigma(i)}\end{array}}}
\prod_{i=1}^\nu|\hat f_0(m_i)|\delta_0\left(\sum\limits_{j=1}^s \lambda_j(\text{\boldmath $e$},\text{\boldmath $r$})\sqrt t_j\right)
\end{align*}
with $\lambda_j(\text{\boldmath $e$},\text{\boldmath $r$}) = \sum\limits_{i\in\sigma^{-1}(j)}e_ir_i$. \\

The square roots of distinct squarefree integers are linearly independent over $\setQ$
so 
$$\sum\limits_{j=1}^s \lambda_j(\text{\boldmath $e$},\text{\boldmath $r$})\sqrt t_j = 0\;\;\Longleftrightarrow \;\; \lambda_j(\text{\boldmath $e$},\text{\boldmath $r$})=\!\!\sum\limits_{i\in\sigma^{-1}(j)}\!\!e_ir_i = 0 \text{ for all } j\in\llbracket1,s\rrbracket.$$

If $|\sigma^{-1}(j)|=1$ for a certain $j\in\llbracket1,s\rrbracket$ then this never holds because the $r_i$ are positive which proves the first assertion. 

Otherwise for fixed $t$ and  for all $j$, we have $|\sigma^{-1}(j)|-1$ degrees of freedom for the choice of  $r_i\le \sqrt{Y^{1+\eta}/t_{\sigma(i)}}$. By \cite[Lemma 3.3]{squarefree}, since $f$ is an eigenform we have
\begin{equation}
\label{boundcoeff}
|\hat f_0(r^2t)|\ll_\epsilon |\hat f_0(t)| r^{\epsilon}+  |\hat f_\infty(t)| r^{\epsilon}
\end{equation}
for any $\epsilon >0$. Thus,

\begin{align*}
&\sum_{\tiny{\begin{array}{c}t_1<\ldots<t_s \\m_i=r_i^2t_{\sigma(i)}\end{array}}}
\prod_{i=1}^\nu|\hat f_0(m_i)|\delta_0\left(\sum\limits_{j=1}^s \lambda_j(e,r)\sqrt t_j \right)\\
&\phantom{xxxxxxxxx}\ll \sum_{t_1<\ldots< t_s} \sum_{\tiny{\begin{array}{c} r_i^2t_{\sigma(i)} <Y^{1+\eta}\\ \sum\limits_{i\in\sigma^{-1}(j)}e_ir_i = 0\end{array}}} \prod_{j=1}^s\prod_{i\in\sigma^{-1}(j)}(|\hat f_0(t_j)|+|\hat f_\infty(t_j)|)r_i^{\epsilon} \\
& \phantom{xxxxxxxxx}\ll Y^{(1+\eta)\epsilon s}\sum_{\text{\boldmath $t$}}\prod_{j=1}^s(|\hat f_0(t_j)|+|\hat f_\infty(t_j)|)^{|\sigma^{-1}(j)|} \left(\sum_{r^2t_j<Y^{1+\eta}}r^\epsilon\right)^{|\sigma^{-1}(j)|-1}
\end{align*}
since for every $j$ we fix the first $|\sigma^{-1}(j)|-1$ values of $r_i$ and we bound the last one (whose value is given by the linear condition) by $Y^{1+\eta}$. Then we get
\begin{align*}
&\sum_{\text{\boldmath $t$}}\prod_{j=1}^s(|\hat f_0(t_j)|+|\hat f_\infty(t_j)|)^{|\sigma^{-1}(j)|}\left(\sum_{r^2t_j<Y^{1+\eta}}r^\epsilon\right)^{|\sigma^{-1}(j)|-1}\\
&\phantom{cccccccccccccc}\ll \sum_{\text{\boldmath $t$}}\prod_{j=1}^s(|\hat f_0(t_j)|+|\hat f_\infty(t_j)|)^{|\sigma^{-1}(j)|} \left(\frac{Y^{1+\eta}}{t_j}\right)^{(1/2+\epsilon)(|\sigma^{-1}(j)|-1)}\\
&\phantom{cccccccccccccc}\ll Y^{(1+\eta)(1/2+\epsilon)( \nu -s)}\prod_{j=1}^s \sum_{1\le t<Y^{1+\eta}}(|\hat f_0(t)|+|\hat f_\infty(t)|)^{|\sigma^{-1}(j)|}t^{-(|\sigma^{-1}(j)|-1)/2}.
\end{align*}

We handle the last sums differently depending on the value of $|\sigma^{-1}(j)|$ (which is an integer $\ge2$). 

$\bullet$ If $|\sigma^{-1}(j)|=2$ then using $(a+b)^2\le 2(a^2+b^2)$ and Lemma \ref{lemme-RS},

\begin{align*}
\sum_{1\le t<Y^{1+\eta}}\frac{(|\hat f_0(t)|+|\hat f_\infty(t)|)^{2}}{\sqrt t}
&\ll \sum_{\mathfrak a\in\{0,\infty\}}\sum_{1\le t<Y^{1+\eta}}\frac{|\hat f_\mathfrak a(t)|^{2}}{\sqrt t}\\
&\ll Y^{(1+\eta)/2}.
\end{align*}

$\bullet$ If $|\sigma^{-1}(j)|=3$ then using \eqref{bound}, we have
\begin{align*}
\sum_{1\le t<Y^{1+\eta}}\frac{(|\hat f_0(t)|+|\hat f_\infty(t)|)^{3}}{t}&\ll Y^{(1+\eta)/6 + \epsilon}\sum_{1\le t<Y^{1+\eta}}\frac{(|\hat f_0(t)|+|\hat f_\infty(t)|)^{2}}{t}\\
&\ll Y^{(1+\eta)/6 + \epsilon} \sum_{\mathfrak a,\mathfrak b\in\{0,\infty\}}\left(\sum_{1\le t<Y^{1+\eta}}\frac{|\hat f_\mathfrak a(t)|^{2}}{t}\right)^{1/2}\left(\sum_{1\le t<Y^{1+\eta}}\frac{|\hat f_\mathfrak b(t)|^{2}}{t}\right)^{1/2}\\
&\ll Y^{(1+\eta)/6+\epsilon}.
\end{align*}

$\bullet$ If $|\sigma^{-1}(j)|\ge 4$ the sums are bounded uniformly in $q$ and $Y$. \\

Thus, we have
\begin{align*}
\Sigma_\sigma(q) &\ll Y^{(1+\eta)(\frac{\nu-s}{2}+\frac{\sigma_2}{2}+\frac{\sigma_3}{6})+\epsilon}\\
&\ll Y^{(1+\eta)(\frac{\nu}{2}-\frac{s-\sigma_2}{3})+\epsilon}
\end{align*}
since $\sigma_2+\sigma_3\le s$ and  for any $\epsilon>0$.
\end{proof}
\vskip 0.8 cm

\begin{rmq}\
\begin{itemize}
\item[$\bullet$] The second term in the right hand-side of \eqref{boundcoeff} comes from the fact that \emph{a priori} $f_0$ is an eigenform only
 for $T_{p^2}$ with $p$ odd. If we have assumed that it is also an eigenform for $T_4$ then we would simply have 
 $|\hat f_0(r^2t)|\ll_\epsilon |\hat f_0(t)|r^\epsilon$.
\item[$\bullet$] In the case of $\setE^-$, we just replace the $m_i$ by $\mu_pm_i$ in the indicator function (take $\mu_p$ positive) and notice that
$$\sum\limits_{i=1}^\nu e_i\sqrt{\mu_p m_i} = 0 \Longleftrightarrow \sum\limits_{i=1}^\nu e_i\sqrt m_i = 0$$
so the same result holds.
\end{itemize}
\end{rmq}

\vskip 0.5 cm \subsection{Odd moments}
We can now answer completely the case of the odd moment.

\begin{theo} 
\label{theo-moment-impair}
Let $\nu$ be an odd positive integer and $C_\nu=\big(2\nu^{2^{\nu-1}}\big)^{-1}$. Assume there exists $\delta > 0$ such that 
$$1 \le Y^{2^{\nu-2}+\delta}<C_\nu q$$
 then for any $\epsilon > 0$

\begin{equation}
\label{momentimpair}
\setE(E_{q}^\nu)= O\left(Y^{-1/3 +\epsilon}+\frac{Y^{\nu/2+\epsilon}}{p}\right)
 \end{equation}
 where the implicit constant only depends on $f, w, \nu, \epsilon$ and $\delta$.
 
 If we  assume that there exists $\delta > 0$ such that 
\begin{equation}
\label{secondgrowth}
1 \le Y^{2^{\nu-2}+\delta}<C_\nu q/p
\end{equation}
 then for any $\epsilon > 0$

\begin{equation}
\label{momentimpair2}
\setE(E_{q}^\nu)= O\left(Y^{-1/3 +\epsilon}+\frac{Y^{\nu/2+\epsilon}}{p^{3/2}}\right)
 \end{equation}
 where the implicit constant only depends on $f, w, \nu, \epsilon$ and $\delta$.
\end{theo}

\begin{proof}
In the first case, the growth  condition on $Y$ enables us to apply Proposition \ref{prop-sigmasigma} for $\Sigma_\sigma(q)$ in \eqref{eq} but we can only apply Proposition \ref{prop-B} to deal with $\Sigma_\sigma(q/p)$  so far. 

Hence we get
$$\setE^+(\M_{q}^\nu)\ll Y^{-1/3+\epsilon} + \frac{Y^{\nu/2+\epsilon}}{p}$$
and the same holds for $\setE^-$ as mentioned. By Proposition \ref{prop-moment}, we reach the conclusion.\\

In the second case, we can apply Proposition \ref{prop-sigmasigma} for both $\Sigma_\sigma(q)$ and $\Sigma_\sigma(q/p)$ in \eqref{eq}. Thus we have 
\begin{align*}
\setE^\pm(\M_{q}^\nu)\ll Y^{-1/3+\epsilon}
\end{align*}
and by Proposition \ref{prop-moment}, we reach the conclusion.

\end{proof}

\begin{rmq}
One can see that the second growth condition  \eqref{secondgrowth} is relevant only for $q=p^N$ and $N>1$.
\end{rmq}

\vskip 0.5 cm \subsection{Even moments}

Next assume that $\nu$ is even. We have only shown that the main term of \eqref{eq1} comes from $\nu$-tuples $m_i=r_i^2t_i$ such that $|\{t_i, i\in\llbracket1,\nu\rrbracket\}|=\nu/2$ so we have to be more precise. Also, we let $S^{*}(\nu,\nu/2)$ denote
the set of all $\sigma\in S(\nu,\nu/2)$ such that $\sigma_2=\nu/2$. 

\begin{prop}
\label{prop-sigmapair}
Let $\nu$ be an even integer and $\sigma\in S^{*}(\nu,\nu/2)$. Assume  there exists $\delta>0$ such that $Y\ll q^{2^{2-\nu}-\delta}$. Then for $\eta$ sufficiently small and $q$ large enough
$$\Sigma_\sigma(q) = 2^{\nu/2} \underset{\tiny{\begin{array}{c}t_1<\ldots<t_{\nu/2} \\m_i=r_i^2t_{i}\end{array}}}{{\sum}^{\:\Box}}\prod_{i=1}^{\nu/2}\hat f_0(m_i)^2B\left(\frac{m_i}{Y}\right)^2 $$

\end{prop}

\begin{proof}
Let $t_1,\ldots,t_{\nu/2}$ be any distinct squarefree integers and let $m_i=r_i^2t_{\sigma(i)}$ as previously. Changing the notation a bit, we may assume that $m_{2i-1}=r_i^2t_i$ and $m_{2i}=r_i'^2t_i$ for $i\in\llbracket1,\nu/2\rrbracket$. Then

$$\sum\limits_{\text{\boldmath $e$}\in\{\pm 1\}^\nu}\delta_q\left(\sum\limits_{i=1}^\nu e_i{\sqrt m_i}^q\right) =
\sum_{\text{\boldmath $e'$}\in\{\pm 1\}^{\nu/2}}\sum_{\text{\boldmath $e$}\in\{\pm 1\}^{\nu/2}}\delta_q\left(\sum\limits_{i=1}^{\nu/2}e_i (r_i+e_{i}'r_i'){\sqrt t_i}^q\right).$$

Fix \text{\boldmath $e$} and apply Lemma \ref{lemme-moment} to $(r_i+e_{i}'r_i')^2t_i$, then for $q$ large enough
$$\sum_{\text{\boldmath $e$}\in\{\pm 1\}^{\nu/2}}\delta_q\left(\sum\limits_{i=1}^{\nu/2}e_i (r_i+e_{i}'r_i'){\sqrt t_i}^q \right)
\le 2^{\nu/2} \sum_{\text{\boldmath $e$}\in\{\pm 1\}^{\nu/2}}\delta_0\left(\sum\limits_{i=1}^{\nu/2}e_i (r_i+e_{i}'r_i'){\sqrt t_i}\right).$$

The right-hand side vanishes except if $\sum\limits_{i=1}^{\nu/2} e_i(r_i+e_{i}'r_i'){\sqrt t_i} = 0$ for some $e$. In this case $r_i+e_{i}'r_i' = 0$ for every $i$ since the $t_i$ are squarefree and distinct. This leads to $r_i=r_i'$ and $e_{i}'=-1$ for every $i$ (because $r_i, r_i'\ge1$). Then we have $m_{2i-1}=m_{2i}$ and in fact

$$\sum\limits_{\text{\boldmath $e$}\in\{\pm 1\}^\nu}\delta_q\left(\sum\limits_{i=1}^\nu e_i{\sqrt m_i}^q\right) = 2^{\nu/2}$$
so we have the desired conclusion.

\end{proof}

We now reach the conclusion for even moments.

\begin{theo} 
\label{theo-momentpair}
Let $\nu$ be an even positive integer and $C_\nu=\big(2\nu^{2^{\nu-1}}\big)^{-1}$. Assume there exists $\delta > 0$ such that 
$$1 \le Y^{2^{\nu-2}+\delta}<C_\nu q$$
 then for any $0<\eta<\delta2^{2-\nu}$ and any $\epsilon > 0$
 
\begin{equation}
\label{momentpair}
\setE^\pm(E_{q}^\nu)=  \frac{\nu !}{(\nu/2)!}\left(\frac{1}{Y}
 \sum_{\tiny{\begin{array}{c}1\le m <Y^{1+\eta}\\ \left(\frac{m}{p}\right)=\pm1\end{array}}}
 \hat f_0(m)^2B^2\left(\frac{m}{Y}\right)\right)^{\nu/2} + O\left(Y^{-1/3 +\epsilon}+\frac{Y^{\nu/2+\epsilon}}{p}\right)
 \end{equation}
where the implicit constant only depends on $f, w, \nu, \epsilon$ and $\eta$.

If we assume that there exists $\delta > 0$ such that 
$$1 \le Y^{2^{\nu-2}+\delta}<C_\nu q/p$$
 then the remainder term in \eqref{momentpair} is $O\left(Y^{-1/3 +\epsilon}+\frac{Y^{\nu/2+\epsilon}}{p^{3/2}}\right)$ where the implicit constant only depends on $f, w, \nu, \epsilon$ and $\eta$.
\end{theo}

\begin{proof}
In the second case, applying twice propositions \ref{prop-sigmasigma} and \ref{prop-sigmapair} for both $\Sigma_\sigma(q)$ and $\Sigma_\sigma(q/p)$ in \eqref{eq}, we have

\begin{align*}
\setE^+(\M_{q}^\nu)=&\;\frac{(1-1/p)^{-1}}{Y^{\nu/2}}\sum_{\sigma\in S^{*}(\nu,\nu/2)} 2^{\nu/2} \underset{\tiny{\begin{array}{c}t_1<\ldots<t_{\nu/2} \\m_i=r_i^2t_{i}\end{array}}}{{\sum}^{\:\Box}}\prod_{i=1}^{\nu/2}\hat f_0(m_i)^2B\left(\frac{m_i}{Y}\right)^2\\
&\phantom{zzz} -\frac{(1-1/p)^{-1}}{pY^{\nu/2}}\sum_{\sigma\in S^{*}(\nu,\nu/2)} 2^{\nu/2} \underset{\tiny{\begin{array}{c}t_1<\ldots<t_{\nu/2} \\m_i=r_i^2t_{i}\end{array}}}{{\sum}^{\:\Box}}\prod_{i=1}^{\nu/2}\hat f_0(m_i)^2B\left(\frac{m_i}{Y}\right)^2 +O\left(Y^{-1/3 + \epsilon}\right)
\end{align*}
and since the quantity in the product is independent of how we enumerate the coefficients $t_i$, we get 
\begin{align*}
Y^{\nu/2}\setE^+(\M_{q}^\nu)= & \;2^{\nu/2}  \frac{|S^{*}(\nu,\nu/2)|}{(\nu/2)!}\underset{\tiny{\begin{array}{c}t_1,\ldots,t_{\nu/2} \text{ distinct} \\m_i=r_i^2t_{i}\end{array}}}{{\sum}^{\:\Box}}\prod_{i=1}^{\nu/2}\hat f_0(m_i)^2B\left(\frac{m_i}{Y}\right)^2 +O\left(Y^{-1/3 + \epsilon}\right)\\
= & \; \frac{\nu !}{(\nu/2)!} \underset{\tiny{\begin{array}{c}t_1,\ldots,t_{\nu/2}  \\m_i=r_i^2t_{i}\end{array}}}{{\sum}^{\:\Box}}\prod_{i=1}^{\nu/2}\hat f_0(m_i)^2B\left(\frac{m_i}{Y}\right)^2 +O\left(Y^{-1/3 + \epsilon}\right)
\end{align*}
since the added terms are negligible by Proposition \ref{prop-sigmasigma}. We get the analogue result for $\setE^-$ (where the sum is taken over non-squares mod $p$) and this leads to the conclusion. 

It is essentialy the same in the first case  since we just have to apply Propositions \ref{prop-sigmasigma} and \ref{prop-sigmapair} for $\Sigma_\sigma(q)$ only and use again that
$$\sum_{m_1,\ldots,m_\nu}\prod_{i=1}^\nu |\hat  f_0(m_i)B\left(\frac{m_i}{Y}\right)|\ll Y^{\nu+\epsilon}.$$

\end{proof}

 \vskip 1 cm \section{Variance}
 \label{section6}
 
 The goal of this section is to give an asymptotic formula for the main term in \eqref{momentpair}, which we can view essentially as the variance of $E_q$. We will distinguish two aspects of convergence. The first one,  when $p\to+\infty$, will lead us to the proof of Corollary \ref{cor-horizontal}. The second one, when $p$ is fixed, will give only a partial result on the sum $S_q$.\\
 
 First note that by a summation by parts, Lemma \ref{lemme-RS} and Lemma \ref{lemme-B}, 

\begin{equation}
\label{YZ}
 \frac{1}{Y} \vline\sum_{m \ge Z} \hat f_0(m)^2B^2\left(\frac{m}{Y}\right)\vline\ll  \left(\frac{Y}{Z}\right)^A
 \end{equation}
for all $A>1$. Therefore, we want to control the following quantity
 
 \begin{equation}
 \label{var}
 \frac{1}{Y} \sum_{\tiny{\begin{array}{c}1\le m < Z\\ \left(\frac{m}{p}\right)=\pm1\end{array}}}
 \hat f_0(m)^2B^2\left(\frac{m}{Y}\right)
 \end{equation}
 when $p$ and $Z$ go to infinity in certain range or when $p$ is fixed and $Z\to+\infty$.
 
 \vskip 0.5 cm \subsection{Variance when $p$ goes to infinity}
 
 We now prove Corollary \ref{cor-horizontal}.  If we omit the condition $ \left(\frac{m}{p}\right)=\pm1$ in the sum in \eqref{var}, using the Mellin inversion formula, Proposition \ref{prop-RS} and Lemma \ref{lemme-plancherel}, we get that the even moments converge in a certain range to 
 \begin{equation}
 \label{variance}
  \frac{\nu !}{(\nu/2)!}\left(\frac{(4\pi)^{\ell+1/2}}{\Gamma(\ell+1/2)}\langle f_0,f_0\rangle ||w||_2^2\right)^{\nu/2}
  \end{equation}
 which is the moment of order $\nu$ of the central gaussian distribution with variance  $$2\frac{(4\pi)^{\ell+1/2}}{\Gamma(\ell+1/2)}\langle f_0,f_0\rangle ||w||_2^2.$$
 
 Otherwise, the presence of a Legendre symbol complicates our task since it twists the sum with a conductor of size $p$, possibly of size much bigger than $Y$ in a certain range (as in \cite{FGKM} and \cite{KR}). \\
 
 Nonetheless, we will show that assuming a more restrictive growth condition on $Z$ and $p$, there exists infinitely many primes $p$ such that 
 
 \begin{equation}
 \label{condileg}
 \legendre{m}=1 \text{ \;\;for every } 1\le m<Z.
 \end{equation}

 \begin{prop}
 \label{prop-Nx}
 Let $Z : [1,+\infty[\to \setR_{\ge0}$ be an unbounded increasing function. We put
 $$N_x=\:\vline\left\{p\le x\:\vline\: \legendre{m}=1 \text{ for all } 1\le m \le Z(p)\right\}\vline.$$
 Then $N_x \underset{x\to +\infty}{\longrightarrow}+\infty$ as long as :
 
 \begin{enumerate}
 \item  $Z \ll \log \log x$ unconditionally,
 \item $Z\le c \sqrt{\log x}$ with $c$ strictly less than an absolute constant if we assume that the Dirichlet $L$-functions $L(s,\chi)$ ($\chi$ real character) do not have a Siegel zero,
 \item $Z\le c\log x$ with $c<1/4$ assuming the Riemann hypothesis for these $L$-functions. 
 \end{enumerate} 
 
 \end{prop}
 
This is a concequence of the Siegel-Walfisz theorem for which we give the following version taken from \cite[Theorem 11.16]{MV}. Point (3) of Proposition \ref{prop-Nx} is deduced from the proof of \cite[Theorem 11.16]{MV}.
 
 \begin{theo}
 \label{theo-SW}
 Let $\chi$ be a non principal real Dirichlet character modulo $q$ and $\psi_\chi(x) = \sum\limits_{p\le x}\chi(p)\Lambda(p)$ where $\Lambda$ is the von Mangoldt function. Then, with assumptions in increasing order of strength as in the statement of Proposition \ref{prop-Nx} :
 
 \begin{enumerate}
 \item $|\psi_\chi(x)| \ll xe^{-c_A\sqrt{\log x}}$ if $q\ll (\log x)^A$, with $c_A$ depending only on $A$.
 \item $|\psi_\chi(x)| \ll  xe^{-c'\sqrt{\log x}}$ if $q\ll e^{2c'\sqrt{\log x}}$ and $c'$ an absolute constant.
 \item $|\psi_\chi(x)| \ll  x^{1/2 +\epsilon}$ for all $\epsilon >0$ and if $q\ll x$.
 \end{enumerate}
 The other implicit constants being absolute.
 \end{theo}

Also, we will need the following classical lemma.

\begin{lemme}
\label{lemme-pipsi}
Let $\chi$ be a Dirichlet character and  $\pi_\chi(x) = \sum\limits_{p\le x}\chi(p)$. Then 
 $$|\pi_\chi(x)|\ll \frac{\max\limits_{t\le x}|\psi_\chi(t)|}{\log x} + \sqrt x.$$
 In particular, $|\pi_\chi(x)|\log x$ satisfies the same inequality as in Theorem \ref{theo-SW}.
\end{lemme}

  \begin{proof}
  First write 
  $$\pi_\chi(x) - \pi_\chi(\sqrt x)= \sum_{\sqrt x< n \le x}\frac{\Lambda(n)}{\log n}\chi(n) - \!\!\!\sum_{\tiny{\begin{array}{c}\sqrt x < p^k \le x \\ k\ge 2\end{array}}}\frac{\chi(p^k)}{k}$$
  then summing by parts, we get the result.

%
%
  \end{proof}

We now prove Proposition \ref{prop-Nx}.
 \begin{proof}[Proof of Proposition \ref{prop-Nx}]
  
  Let $g(t)=2^{\pi(Z(t))}$ and  denote by $(p_i)_i$ the prime number sequence. We have 
  
  \begin{align*}
  N_x&= \sum_{p\le x}\prod_{p_i\le Z(p)}\frac{1}{2}\left(1+\legendre{p_i}\right)\\
  &=  \sum_{p\le x}\frac{1}{g(p)}\sum_{I\subset \llbracket1;\pi(Z(p))\rrbracket}\legendre{q_I} \text{   \:\:\:\:with\:  }  q_I=\prod_{i\in I}p_i\\
  & = \sum_{I\subset \llbracket1;\pi(Z(x))\rrbracket} \sum_{\tiny{\begin{array}{c}p\le x \\
   \max I \le \pi(Z(p))\end{array}}}\frac{1}{g(p)} \legendre{q_I}\\
   &= \sum_{p\le x} \frac{1}{g(p)} + \sum_{\emptyset \neq I\subset \llbracket1;\pi(Z(x))\rrbracket} \sum_{\tiny{\begin{array}{c}p\le x \\
   \max I \le \pi(Z(p))\end{array}}}\frac{1}{g(p)} \legendre{q_I}
  \end{align*} 

A summation by parts gives $$\sum_{m\le j \le n}\frac{1}{g(p_j)} \left(\frac{q_I}{p_j}\right) = \frac{\pi_I(p_n)}{g(p_n)} -  \frac{\pi_I(p_{m-1})}{g(p_{m-1})} + \sum_{m\le j< n}\left(\frac{1}{g(p_j)}-\frac{1}{g(p_{j+1})}\right)\pi_I(p_j)$$   
 with $\pi_I(t) = \sum\limits_{p\le t}\legendre{q_I}$ and of course $\pi_\emptyset = \pi$.
 
 The main term of $N_x$ is greater than $\frac{\pi(x)}{g(x)}$ and the other terms are bounded by 
 $O\left(\max\limits_{t\le x}|\pi_I(t)|\right)$. Moreover, the log of the conductor of $ \left(\frac{q_I}{\cdot}\right)$ is bounded by $$\sum_{p\le Z(x)}\log p \le (1+\epsilon) Z(x) \text{\:\:\:  for $x$ large enough}.$$
 
 Thus,
 $$N_x \ge \frac{\pi(x)}{g(x)} + O\left(g(x)\max\limits_{t\le x, I}|\pi_I(t)|\right).$$
 According to Theorem \ref{theo-SW} and Lemma \ref{lemme-pipsi}, if $Z(x)\le \frac{A}{1+\epsilon}\log\log x$ then $g(x)$ and $q_I$ are $\le (\log x)^{A}$ so we have  
 $$N_x \ge \frac{\pi(x)}{(\log x)^A}\left(1+ O\left((\log x)^{2A}e^{-c_A\sqrt{\log x}}\right)\right) \to +\infty.$$
 
 In the second case (\textit{i.e.} we assume there is no exceptional zero for $L(s,\chi)$), with $c<c'/2$ we get $q_I\le e^{2c'\sqrt{\log x}}$ and $|\pi_I(x)| \ll \frac{x}{\log x}e^{-c'\sqrt{\log x}}$ so
 
  $$N_x \ge \pi(x)e^{-c\sqrt{\log x}}\left(1+ O\left(e^{2c\sqrt{\log x}-c'\sqrt{\log x}}\right)\right) \to +\infty.$$

In the last case (\textit{i.e.} under GRH), we have
$$N_x \ge \frac{\pi(x)}{x^{c}}\left(1+ O\left(x^{2c - 1/2 +\epsilon}\log x\right)\right) \to +\infty$$
taking $\epsilon$ small enough.
 \end{proof}
 
 \begin{rmq}
 Unfortunately we do not obtain a positive density of primes $p$ satisfying \eqref{condileg} since
 \begin{align*}
 \frac{N_x}{\pi(x)}&\ll \frac{1}{\pi(x)}\sum_{\sqrt x < p\le x}\frac{1}{g(p)} + \frac{x^{1/2}}{\pi(x)} + o(1) \\
                                & \ll \frac{1}{g(\sqrt x)} + o(1) \rightarrow 0
 \end{align*}
 in any case.
 \end{rmq}
 
We now explain how we get \eqref{variance} in the case where there is no twist by a Legendre symbol in \eqref{var}.

 \begin{prop}
 \label{prop-variance}
  Assume that the coefficients $(\hat f_0(n))_{n\ge1}$ are all real numbers, then
 $$\frac{1}{Y} \sum_{1\le m <Z} \hat f_0(m)^2B^2\left(\frac{m}{Y}\right) = \frac{(4\pi)^{\ell+1/2}}{\Gamma(\ell+1/2)}\langle f_0,f_0\rangle ||w||_2^2+ o(1).$$
 when $Y, Z\to+\infty$ and $Y=o(Z)$.
 \end{prop}
 
 \begin{rmq}
 Note that the assumption that the Fourier coefficients are real is only used here. Nonetheless it is a crucial assumption.
 \end{rmq}
 
 \begin{proof}
  We have $\hat f_0(m)^2= |\hat f_0(m)|^2$ so from \eqref{YZ} we get
 \begin{align*}
 \frac{1}{Y} \sum_{1\le m <Z} \hat f_0(m)^2B^2\left(\frac{m}{Y}\right)
 & = \frac{1}{Y} \sum_{m \ge1} |\hat f_0(m)|^2B^2\left(\frac{m}{Y}\right) +o(1)\\
 &=\frac{1}{2i\pi}\int_{(\sigma)}D(s, f_0\times \bar f_0)Y^{s-1} \widehat{B^2}(s)ds+o(1) \text{  \;\; for $\sigma>1$}\\
 &= \frac{(4\pi)^{\ell+1/2}}{\Gamma(\ell+1/2)}\langle f_0,f_0\rangle\widehat{B^2}(1)+o(1)
 \end{align*}
 shifting the domain of integration and using the residue theorem. Applying Lemma \ref{lemme-plancherel}, we obtain the result.
 \end{proof}
 
 We have then essentially shown Corollary \ref{cor-horizontal} since under its assumptions, we have $Y_k=q_k^2/X_k\ll_\epsilon p_k^\epsilon$ and we can apply Theorem \ref{theo-moment-impair}  for odd moments. For even positive $\nu$, by  Theorem \ref{theo-momentpair}, \eqref{plusmoins} and \eqref{YZ} we get 
  $$\setE(E_{q_k}^\nu)=\frac{1}{2} \frac{\nu !}{(\nu/2)!}\left(\frac{1}{Y}\!\!\!
 \sum_{\tiny{\begin{array}{c}1\le m <Z(p_k)\end{array}}}\!\!\!
 \hat f_0(m)^2B^2\left(\frac{m}{Y}\right)+o(1)\right)^{\nu/2}+ \;\;\frac{1}{2}\times 0 +o(1)$$
 with  $Z(p_k)=\log\log p_k$ and $p_k$ satisfies \eqref{condileg}. By Proposition \ref{prop-Nx}, there are infinitely many such primes so
 the moment  converges to $\frac{1}{2} \frac{\nu !}{(\nu/2)!}\left(\frac{(4\pi)^{\ell+1/2}}{\Gamma(\ell+1/2)}\langle f_0,f_0\rangle ||w||_2^2\right)^{\nu/2}$ by Proposition \ref{prop-variance} for a subsequence of prime numbers.
 
 Hence, there exists a subsequence of $q_k$ such that the moments of $E_{q_k}$ converge to the moments of the following distribution
 \begin{equation}
 \label{distribution}
 \frac{1}{2}\delta_0+\frac{1}{2}\N(0,2V_{f,w}).
\end{equation}
 
 Since the moment generating function exists for a normal distribution, it exists for the distribution in \eqref{distribution} which
  is then determined by its moments. Thus, the convergence of the moments of $E_{q_k}$ implies the convergence in 
 law (see for example \cite[\S 5.8.4]{Gut}).
 
 We also see how the range condition of Corollary \ref{cor-horizontal} can be relaxed depending on which case  
 of Proposition \ref{prop-Nx} we look at.
 
  \vskip 0.5 cm \subsection{Variance for a fixed prime}
 
In general, if we fix an odd prime number $p$, we can write

 $$\sum_{\tiny{\begin{array}{c}1\le m <Z \\ \left(\frac{m}{p}\right)=\pm1\end{array}}}
 \hat f_0(m)^2B^2\left(\frac{m}{Y}\right)=\frac{1}{2}\sum_{m<Z} \hat f_0(m)^2B^2\left(\frac{m}{Y}\right)\pm \frac{1}{2}\sum_{m<Z}\legendre{m} \hat f_0(m)^2B^2\left(\frac{m}{Y}\right).$$
 
 If $\chi$ is a real Dirichlet character of conductor $p$, it follows from \cite[Proposition 3.12]{Ono} that 
 $$f_{0,\chi}(z)=\sum_{n\ge1}\chi(n)\hat f_0(n) n^{\ell/2-1/4}e(nz)$$
 is a cusp form of weight $\ell+1/2$ but only for the congruence subgroup $\Gamma_0(4p^2)$. Yet, Proposition \ref{prop-RS} holds for these forms (recall that $f_0$ is \textit{a fortiori} such a form) replacing $\Gamma_0(4)$ by $\Gamma_0(4p^2)$ and $\langle f,g\rangle$ by 
 $$\langle f,g\rangle_p :=\frac{1}{{\rm Vol}(\Gamma_0(4p^2)\backslash\calH)}\int_{\Gamma_0(4p^2)\backslash\calH}f(z)\overline{g(z)}y^{\ell+1/2}\frac{dxdy}{y^2}$$
 and thus we get an analogue of Proposition \ref{prop-variance} : 
  $$\frac{1}{Y} \sum_{1\le m <Z} \chi(m)\hat f_0(m)^2B^2\left(\frac{m}{Y}\right) = \frac{(4\pi)^{\ell+1/2}}{\Gamma(\ell+1/2)}\langle f_0,f_{0,\chi}\rangle_p ||w||_2^2+ o(1).$$
 for $\chi=\legendre{\cdot}$ when $Y, Z\to+\infty$ and $Y=o(Z)$.\\
 
 So fixing $p$ but taking $Y, q=p^N\to+\infty$ with $Y\ll_\epsilon q^\epsilon$ for any $\epsilon$, we obtain the convergence of all moments of the quantity in \eqref{sommemoins} and which, under these assumptions, converges in law to the mixed distribution 
$$\frac{1}{2}\N(0,V_{f,w}+ V_{f,w}^{(p)}) + \frac{1}{2}\N(0,V_{f,w}- V_{f,w}^{(p)})$$
with $ V_{f,w}^{(p)}=\frac{(4\pi)^{\ell+1/2}}{\Gamma(\ell+1/2)}\langle f_{0,\chi}, f_0\rangle_p ||w||_2^2$ and $\chi =\legendre{\cdot}$.

\vskip 1.5 cm \section{About the integral weight case}

In fact, our work can be applied to compute the moments of Fourier coefficients of integral weight modular
 forms in arithmetic progressions of modulus $q=p^N$ with $p$ prime and $N>1$. We will not recall the basic
 definitions and properties of these forms but a good introduction to the whole theory can be found in
  \cite{DiaShu}. Let
 $$f(z)=\sum_{n\ge1}a(n)n^{(\kappa-1)/2}e(nz)$$ for any $z\in\calH$, be a holomorphic cusp form of level 1 and even weight 
 $\kappa$.

Let us recall that if $f$ is an eigenform (\emph{i.e.} an eigenvector for all the Hecke operators) then its normalized Fourier coefficients satisfy  Deligne's bound (see \cite{weil1})
\begin{equation}
\label{deligne}
|a(n)|\ll_{f,\epsilon}n^\epsilon
\end{equation}
for any positive integer $n$ and any $ \epsilon>0$.\\

For any smooth and compactly supported function $w : ]0,+\infty[\to\setR_{\ge 0}$, any $X>0$, any odd prime power $q=p^N$ and any integer $a$ coprime to $q$, we define once again the following
\begin{align}
\label{defS2}
 S(X,q,a)&=\sum_{n=a\:[q]}a(n)w\left(\frac{n}{X}\right),\\
\label{defE2}
E(X,q,a)&=\frac{S(X,q,a)}{\sqrt{X/q}}.
\end{align}

When $q$ is not a prime (so we are not in the conditions of Theorem \ref{theo-FGKM}) we obtain an analogue of Theorem \ref{theo-moment} which implies of course Theorem \ref{theo-moment2intro}.
\begin{theo} 
\label{theo-moment2}
Let $f$ be a Hecke eigenform of level 1 and of even weight $\kappa$. Let $q=p^N$ be an odd prime power with $N>1$ and let $X>0$. Define $E(X,q,a)$ as in \eqref{defE2}.
Let $\nu$ be a positive integer, $C_\nu=\big(2\nu^{2^{\nu-1}}\big)^{-1}$ and $e\in\{\pm1\}$. 
Assume there exists $\delta>0$ such that
$$1 \le Y^{2^{\nu-2}+\delta}<C_\nu q$$
 where $Y=q^2/X$. Then for any $0<\eta<\delta 2^{2-\nu}$ and any $\epsilon>0$, we have
 \begin{align}
 \frac{2}{\varphi(q)}\!\!\!\sum_{\tiny\begin{array}{c}a\:[q]\\ \legendre{a}=e\end{array}}\!\!\!E(X,q,a)^\nu &=
   \delta_{2\mid\nu}\frac{\nu !}{(\nu/2)!}\left(\frac{1}{Y}\sum_{\tiny{\begin{array}{c}1\le m <Y^{1+\eta}\\ \left(\frac{m}{p}\right)=e\end{array}}}
 a(m)^2B^2\left(\frac{m}{Y}\right)\right)^{\nu/2}\\
 &\phantom{aaaaaaaaaaaaaaaaaaaaaaaaaaaa} + O\left(Y^{-1/2 + \epsilon}+\frac{Y^{\nu/2+\epsilon}}{p}\right)
 \nonumber
 \end{align}
where $(a(m))_{m\ge 1}$ are the  normalized Fourier coefficients of $f$ and $B$ is a smooth rapidly decreasing function depending only on $w$ and $k$. 
\end{theo}

We will only give the sketch of the proof for the moment over the squares  modulo $q$ since it is basically the same as for Theorem \ref{theo-moment}. The computational details are even easier using \eqref{deligne}.

\begin{proof}
First we give the following Vorono\u\i\: formula for a classical cusp form. This version is equivalent to the one given in \cite[Theorem 2.2]{godber}. For $(b,q)=1$ we have

$$\sum_{n\ge1}a(n)e_q(bn)w\left(\frac{n}{X}\right)=i^\kappa
\frac{X}{q}\sum_{m\ge1}a(m)e_q(-\bar b m)B\left(\frac{m}{q^2/X}\right)$$
where $$B(x)=\int_{(\sigma)}\frac{\Gamma_f(s)}{\Gamma_f(1-s)}\hat w(1-s)x^{-s}ds$$
for $x>0$, $\sigma > -(\kappa-1)/2$ and $\Gamma_f(s)=(2\pi)^{-(s+(\kappa-1)/2)}\Gamma(s+(\kappa-1)/2)$. 
Note that $B$ satisfies the same conditions as in Section 3.\\

Thus, as in the proof of Proposition \ref{prop-moment}, we get
\begin{align}
\label{calculE}
E(X,q,a)=\frac{i^\kappa}{\sqrt Y}\sum_{m\ge1}a(m)\text{Kl}_q(m,a)B\left(\frac{m}{Y}\right) + O(Y^{\nu/2+\epsilon}p^{-3/2})
\end{align}
so we have to estimate the moment of Kloosterman sums. For $N>1$, these sums are very similar to Salié sums. Precisely,  let $m_1,\ldots,m_\nu$ be any positive integers such that $\legendre{m_i}=1$ for any $i\in\llbracket1,\nu\rrbracket$. From Proposition \ref{prop-Salie}, we have for even $N$,
$$\setE^+\left(\prod_{i=1}^\nu \text{Kl}_q(m_i,\cdot)\right)=
\sum\limits_{\text{\boldmath $e$}\in\{\pm 1\}^\nu}\delta_q\left({\sum\limits_{i=1}^\nu e_i{\sqrt{m_i}}^q}\right) + O(1/p). $$

While if $N$ is odd and $p=1\:[4]$,
$$\setE^+\left(\prod_{i=1}^\nu \text{Kl}_q(m_i,\cdot)\right)=\sum\limits_{\text{\boldmath $e$}\in\{\pm 1\}^\nu}\legendre{{\sqrt{m_1\cdots m_\nu}}^q}G\left(\left(\frac{\cdot}{q}\right)^\nu, \sum_{i=1}^\nu e_i{\sqrt{m_i}}^q\right)+  O(1/p) $$
with $G(\chi,a)=\frac{1}{q}\sum\limits_{x\:[q]}\chi(x)e_q(ax)$ is the classical Gauss sum for any multiplicative character $\chi$ and any $a$ modulo $q$.\\

Finally,  if $N$ is odd and $p=3\:[4]$,
$$\setE^+\left(\prod_{i=1}^\nu \text{Kl}_q(m_i,\cdot)\right)=i^\nu\sum\limits_{\text{\boldmath $e$}\in\{\pm 1\}^\nu}\legendre{{e_1\sqrt{m_1}^q\cdots e_\nu{\sqrt m_\nu}}^q}G\left(\left(\frac{\cdot}{q}\right)^\nu, \sum_{i=1}^\nu e_i{\sqrt{m_i}}^q\right)+  O(1/p). $$
\vskip 0.6cm
If $N$ is even then the proof is exactly the same. Note that we do not need Lemma \ref{lemme-RS} to get an analogue of Proposition \ref{prop-sigmasigma} since \eqref{deligne} is sufficient. Actually, we even get a better remaining term.\\

If $N$ is odd then for odd $\nu$ and any $a$ mod $q$ we have $G\left(\left(\frac{\cdot}{q}\right)^\nu, a\right)\ll q^{-1/2}\le 1/p$ since $N>1$. So we trivially bound the moment.\\

If $N$ is odd, $\nu$ is even and $p=3\:[4]$ (the case $p=1\:[4]$ is similar) then we have 

\begin{align*}
&\phantom{zzzzz}\setE^+\left(E(X,q,\cdot)^\nu\right)=\\
&\phantom{zzzzzzzzzzz}\frac{i^\nu}{Y^{\nu/2}}\sum\limits_{\text{\boldmath $e$}\in\{\pm 1\}^\nu}\underset{\tiny 1\le m_1,\ldots,m_\nu < Y^{1+\eta}}{{\sum}^\Box}\prod_{i=1}^\nu \legendre{e_i{\sqrt m_i}^q}a(m_i)B\left(\frac{m_i}{Y}\right)\delta_q\left({\sum\limits_{i=1}^\nu e_i{\sqrt{m_i}}^q}\right)\\
&\phantom{zzzzzzzzzzzz}+O\left(\frac{Y^{\nu/2+\epsilon}}{p}+Y^{-A}\right).
\end{align*}

As in Proposition \ref{prop-sigmapair}, we prove that the main term comes from the tuples {\boldmath$m$}  which take exactly $\nu/2$ values and with exactly half of the $e_1,\ldots,e_\nu$ being negative.\\

Hence, we reach the desired conclusion.

\end{proof}

Corollary \ref{cor-horizontal2} follows easily from Section \ref{section6}, adapting the discussion to an integral weight cusp form.

 \newpage


\end{document}